\documentclass[11pt,numbook,runningheads]{svjour3}     % onecolumn (second format)
\usepackage{amsmath,graphicx,amssymb,enumerate,bm,url}

\newcommand{\bs}[1]{\boldsymbol{#1}}

\def \ri {{\rm i}}
\smartqed  % flush right qed marks, e.g. at end of proof
\usepackage{algorithm}
\usepackage{algpseudocode}

\usepackage{caption}
\usepackage{subfigure}

\usepackage{amssymb}
\usepackage{latexsym}
\usepackage{amsmath,amssymb}
\usepackage{amsfonts}
\usepackage{bm}
\usepackage{geometry}
\usepackage{graphicx}
\usepackage{amsmath}
\usepackage{epstopdf} % needed if you have eps figures in a pdflatex manuscript
\usepackage{mathptmx}      % use Times fonts if available on your TeX system
\usepackage{amsfonts}
\usepackage{mathrsfs}
\usepackage{amssymb}
\usepackage{amsmath}
\usepackage{epsfig}
\usepackage{float}
\usepackage{caption}
\usepackage{bm}
\usepackage{mathrsfs}
\usepackage{float}
\usepackage{geometry}
 \journalname{Journal of Scientific Computing}
\begin{document}

\title{Fast multipole method for the Laplace equation in half plane with Robin boundary condition\thanks{The first author was supported by Postgraduate Scientific Research Innovation Project of Hunan Province (No. CX20230512). The second author was  supported by NSFC (grant 12022104, 12371394) and the Major Program of Xiangjiang Laboratory (No.22XJ01013). The third author was  supported by NSFC (grant 12201603). The fourth author was  supported by  Clements Chair for this research. }}

%\titlerunning{Asymptotic Analysis and Numerical methods for Infinite Generalized Bessel Transforms}        % if too long for running head

\author{Chun zhi Xiang  \and Bo Wang\and Wenzhong Zhang\and Wei Cai%etc.
}


\institute{Chun zhi Xiang \at LCSM, Ministry of Education,  School of Mathematics and
Statistics, Hunan Normal University, Changsha, Hunan 410081, P. R. China. \\
\email{chunzhixiang@hunnu.edu.cn@163.com}           %  \\
\and
Bo Wang \at Corresponding author, LCSM, Ministry of Education, School of Mathematics and
Statistics, Hunan Normal University, Changsha, Hunan 410081, P. R. China.
Department of Mathematics, Southern Methodist University, Dallas, TX 75275. \\
\email{bowang@hunnu.edu.cn}
\and 
Wenzhong Zhang \at Suzhou Institute for Advanced Research, University of Science and Technology of China, Jiangsu 21500, P. R. China.\\
\email{wenzhongz@mail.smu.edu}   
\and  
Wei Cai \at Department of Mathematics, Southern
Methodist University, Dallas, TX 75275.\\
\email{cai@smu.edu}  }  

\date{Received: date / Accepted: date}

\maketitle

\begin{abstract}
In this paper, we present a fast multipole method (FMM) for solving the two-dimensional Laplace equation in a half-plane with Robin boundary conditions. The method is based on a novel expansion theory for the reaction component of the Green's function. By applying the Fourier transform, the reaction field component is obtained in a Sommerfeld-type integral form. We derive far-field approximations and corresponding shifting and translation operators from the Fourier integral representation. The FMM for the reaction component is then developed by using the new far-field approximations incorporated into the classic FMM framework in which the tree structure is constructed from the original and image charges. Combining this with the standard FMM for the free-space components, we develop a fast algorithm to compute the interaction of the half plane Laplace Green's function. We prove that the method exhibits exponential convergence, similar to the free-space FMM. Finally, numerical examples are presented to validate the theoretical results and demonstrate that the FMM achieves $O(N)$ computational complexity.
\end{abstract}

\section{Introduction}
The Laplace equation in half space $\mathbb R^2_+=\{\bs r=(x, y): y>0\}$ with Robin boundary condition on the boundary $y=0$, has a wide array of practical implications in engineering contexts. An important application is in water wave theory when the assumption of infinite depth is applied as long as the water is deep enough with respect to the wave height and length \cite{Hein1,Kuznetsov,appliquees2010greens,Perez}. The real-parameter Robin boundary condition offers a linearized depiction of time-harmonic gravity wave propagation across the surface of incompressible, inviscid, and irrotational fluids \cite{Wehausen}.  Robin boundary condition with complex-parameter will be employed when porous structures such as permeable breakwaters are considered \cite{Sollitt}. By allowing local perturbations on the half-plane, the model is extended to describe the scattering of small-amplitude water waves due to the presence of floating or submerged bodies. 
Other notable applications include the modeling of harmonic potentials within domains featuring uneven surfaces, the analysis of steady-state heat conduction employing linear convective boundary conditions, and the approximation of low-frequency sound waves and electromagnetic wave propagation on the ground \cite{chen1997review,dassios1999half}.

For numerically addressing the Laplace equation in a Robin half-plane,  the boundary integral method \cite{Mei,Perez,Yeung} has the advantages of dimension reduction and naturally imposing the radiation/decay condition. Nevertheless, a well-known limitation of the conventional boundary integral method resides in the dense linear system resulting from the discretization of the global boundary integral operator. Solving this linear system using standard methods can present computational challenges, especially in scenarios involving intricate or extensive boundaries. One strategy to circumvent this hindrance involves the fast multipole method (FMM), originally devised by Rokhlin \cite{Rokhlin1985} for the two-dimensional Laplace equation and subsequently refined by Greengard and Rokhlin for many-body issues \cite{Greengard1}. This method accelerates the dense matrix product vector, reduces data storage requirements, and decreases the computational cost from $O(N^2)$ to $O(N\log N)$ or $O(N)$. Over the past three decades, there has been a significant body of research focusing on the FMM \cite{widefmm,Darve1,darve2000fast,greengard1997new,Lu1,snyder12,Song1,Tausch1,bo2018taylorfmm,Ying1} and its applications on solving PDEs with boundary integral methods \cite{Fong1,Liu1}. 
% The multipole expansion (ME), local expansion (LE), and multipole to local translation (M2L) constitute the mathematical framework of fast multipole method (FMM) used to evaluate integral operators associated with the Green's function in wave scattering scenarios \cite{widefmm, Lu1,rokhlin1990rapid,snyder12,Sochacki91interface,Song1,wangbo2021,Zhang-Wang}.
%Moreover, several versions of the fast multipole method (FMM) have been proposed based on plane wave expansion \cite{Darve1,Liu1}, Taylor expansion \cite{Tausch1,bo2018taylorfmm} and kernel independent compression techniques \cite{Fong1,Ying1}.

The essence of the FMM is the far field approximation of the Green's function. Unlike the half-plane problem with Dirichlet or Neumann boundary conditions where a closed form of the Green's function can be obtained by simply applying image method, the Green's function for the Robin problem is usually given by more complicate expressions. Its far field approximation theory has not been established until the first work by Hien et al. \cite{Hein1,appliquees2010greens} in which a closed form of the Green's function has been derived. Based on this closed form, a FMM accelerated boundary element method (BEM) is developed to efficiently solve the half-plane problem \cite{Perez}. However, to the best of our knowledge, the closed form of the Green's function for 3-dimensional half space problem is still open and the road-map presented in \cite{Hein1,appliquees2010greens} is not available in handling half space problems. 

Recently, we have established a general framework to develop FMM for the Green's function of 3-dimensional Laplace, Helmholtz and modified Helmholtz equation in layered media \cite{wang2019fast,wangbo2021,wang2021fast,Zhang-Wang}. We have proposed a methodology to derive far field expansion theory from the Sommerfeld-type integral representation of the Green's functions in layered media.  The derivation is based on the expansions of the Fourier kernel and thus can be applied to a variety of linear PDEs whose Green's function can be obtained through Fourier transform. Moreover, the resultant expansions reduce to the spherical harmonic expansions used in the free space FMM when the layered medium is reduced to the homogeneous one. In all of our previous work, only transmission interface conditions are considered. Robin boundary condition will lead to singular density function in the Sommerfeld-type integral representation of the Green's function and therefore worthy further investigation.  

% Although the fast multipole method (FMM) has been extensively applied in recent years, research on the algorithm's error remains limited, particularly concerning the overall error. Most studies rely on numerical experiments to demonstrate the convergence of the algorithm's error, with few focusing on theoretical investigations. Wala and Kl{\"o}ckner provided a theoretical analysis of the overall error of the fast multipole method for two-dimensional potential problems \cite{Wala1}. Furthermore, an estimation of the overall error for the fast multipole method in two-dimensional scattering problems is presented in \cite{Meng1}, including clear error bounds and convergence rates. Xiang and Liu \cite{Xiang1} derived the optimal convergence rates for a two-dimensional fast multipole method.

In this article, we develop an FMM for the Laplace equation in half-plane with Robin boundary condition. No closed form of the Green's function is required. We present a comprehensive derivation for the far field expansions and their shifting and translation operators of the Green's function directly from the Sommerfeld-type integral representation. Although the approximation theory used in our algorithm might be equivalent to that presented in \cite{Hein1,appliquees2010greens}, the derivation procedure is much simpler and different. More importantly, our framework is more general and can be extended to 3-dimensional half-space case naturally, which is  our on-going research. Exponential convergence of the far field expansions and their shifting and translation operators are proved. The result reveals an important fact that the convergence of the approximations used for the reaction field component depends on the distance between the target and the image source. This suggests how the fast multipole method (FMM) framework should be configured for sources and targets located in the half-plane.

The outline of the subsequent sections is as follows. In Section \ref{sec2}, we first present the derivation of the expansion theory for the reaction component of the Green’s functions of Laplace equation in the half-plane. Then, the FMM for the reaction components is developed. Together with the classic FMM for other free space components, a fast algorithm for the Laplace equation in the half-plane with Robin boundary condition is obtained. The exponential convergence of the FMM is proven in  Section \ref{section4}. The theoretical analysis shows that the FMM for the reaction component has better convergence as the rate depends on the distance between the original sources and the images of the targets which are always separated by the boundary $y=0$. Numerical experiments are provided in  Section \ref{section5} to validate our theoretical analysis and the $O(N)$ complexity of the FMM. Finally, we summarize this paper and discuss future work in  Section \ref{section6}.

\section{The FMM for Laplace equation in half-plane}\label{sec2}
In this section, we first present mathematical expansions for the far-field approximation of the reaction components in the Green’s function for the two-dimensional (2-D) Laplace equation in a half-plane domain. Subsequently, we introduce the FMM for the Green’s function of the 2-D Laplace equation in a half-plane, utilizing the framework of FMM in layered media.
\subsection{The Green's function for Laplace equation in half-plane}\label{section2}
We focus on the radiation problem of linear time-harmonic surface waves in the half-plane $\mathbb{R}_+^2$, originating from a fixed source point $\bs r' \in \mathbb{R}_+^2$, as illustrated in Fig. \ref{wave}. The Green's function discussed herein corresponds to the solution of this problem. Our analysis considers two distinct cases of wave propagation. The first case involves dissipative wave propagation, described by the Green function that incorporates dissipation. In contrast, the second case pertains to non-dissipative wave propagation, represented by the Green function without dissipation. To ensure correct physical results, the calculations must be conducted in the sense of the limiting absorption principle, which states that the Green's function without dissipation is the limit of the Green's function with dissipation as the dissipation parameter $\varepsilon$ approaches zero.
\begin{figure}[h!]
	\centering
	\includegraphics[width=0.75\textwidth]{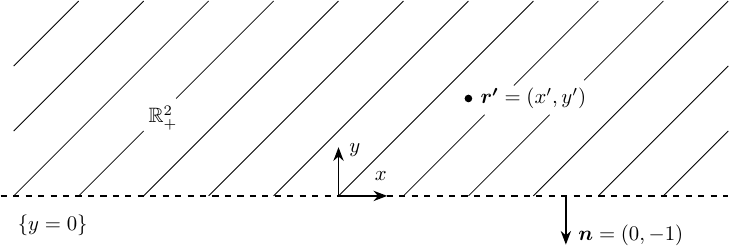}
	\caption{Domain of the Green's function of the half-plane problem.}
	\label{wave}
\end{figure} 

The Green's function for the half-plane problem in the two-dimensional Laplace equation (taking the upper half-plane as an example) satisfies the following conditions:
\begin{eqnarray}
	\label{0.0}	\bs{\Delta}G(\bs r, \bs r')=-\delta(\bs r-\bs r'), \quad y>0,\label{helmholtztotal}\\
	\label{1.0}	\mathscr{B}G=0,\quad y=0,\\
	\label{1.1} |G| \leq \frac{C}{{r}}, \quad \left|\frac{\partial G}{\partial r}\right| \leq \frac{C}{r^{2}}, \quad r \to \infty,
\end{eqnarray}
where
\begin{equation}
	\mathscr{B}=\mathcal I,\;\; \frac{\partial}{\partial\bs n},\;\;{\rm or}\;\;\frac{\partial}{\partial\bs n}-Z_{\varepsilon},
\end{equation}
are boundary operators regarding Dirichlet, Neumann, and Robin boundary conditions, respectively. Here, $\bs r=(x,y)$, $\bs r'=(x',y')$, $r=|\bs r|$, and $C$ is a positive constant. The outward normal vector $\bs n$ is specified as the negative direction of the $y-$axis. The notation $Z_{\varepsilon}=Z+\ri \varepsilon$ denotes a complex impedance that corresponds to dissipative wave propagation. Moreover, $\varepsilon>0$ is a small dissipation parameter and $Z>0$. If $Z_{\varepsilon}=Z$, this indicates a real impedance, which is associated with non-dissipative wave propagation.
The inequalities in \eqref{1.1} represent the outgoing radiation condition at infinity related to dissipative wave propagation. For non-dissipative wave propagation, an outgoing radiation condition at infinity was introduced in reference \cite{Hein1}. This condition, as $r \to \infty $, is expressed as follows,
\begin{align}
	|G| \leq \frac{C}{{r}} \quad &\text {and}\quad \left|\frac{\partial G}{\partial r}\right| \leq \frac{C}{r^{2}} \quad  &\text {if} \quad y>\frac{1}{Z} \ln (1+Z \pi r), \label{c1.1}\\
	|G| \leq C \quad &\text {and} \quad \left|\frac{\partial G}{\partial r}-i Z G\right| \leq \frac{C}{r} \quad &\text {if}\quad y<\frac{1}{Z} \ln (1+Z \pi r),\label{c1.2}
\end{align}	
for some constants $C>0$.

By applying Fourier transform in the $x$-direction and the idea of images, we can obtain expressions for the Green's functions associated with different boundary conditions as follows: 
\begin{itemize}
	\item Dirichlet and Neumann boundary conditions
	\begin{eqnarray}
		\label{ddd1.2}G(\bs r, \bs r')=-\dfrac{1}{2\pi}\ln|\bs r-\bs r'|\pm\dfrac{1}{2\pi}\ln|\bs r-{\bs r}'_{\rm im}|,
	\end{eqnarray}
	\item Robin boundary condition
	\begin{equation}\label{ddGreensfunimp}
		G(\bs r, \bs r')=-\dfrac{1}{2\pi}\ln|\bs r-\bs r'|+\dfrac{1}{2\pi}\ln|\bs r-{\bs r}'_{\rm im}|
		+\frac{1}{2\pi}\int_{-\infty}^{+\infty}\frac{e^{\ri \lambda(x-x')-|\lambda|(y+y')}}{|\lambda|-Z_{\varepsilon}}d\lambda,
	\end{equation}
\end{itemize}
where
\({\bs r}'_{\rm im}=(x', -y')\) is the image of \(\bs r'\) with respect to the \(x\)-axis. Apparently, the Green's function has symmetry 
\begin{equation}\label{greensfunsymm}
	G(\bs r, \bs r')=G(\bs r', \bs r),
\end{equation}
for all three types of boundary conditions. In the Robin boundary condition case, only the third term in  \eqref{ddGreensfunimp} depends on the impedance $Z_{\varepsilon}$. In this paper, we will present a fast multipole method for fast computation of the interactions induced by the Green's function \eqref{ddGreensfunimp}. As the classic FMM can be applied to calculate the interactions induced by the logarithm terms, we will focus on the theory and fast algorithms for the computation of the interaction governed by the integral term, which we denoted by
\begin{equation}
	\label{dzGreensfunimp} 
	G_{Z_{\varepsilon}}(\bs r, \bs r'):=\frac{1}{2\pi}\int_{-\infty}^{+\infty}\frac{e^{\ri \lambda(x-x')-|\lambda|(y+y')}}{|\lambda|-Z_{\varepsilon}}d\lambda.
\end{equation}
The approximation theory and fast algorithm presented in this paper is also available for the non-dissipative case, i.e., $\varepsilon=0$. In this case, the integral \eqref{dzGreensfunimp} involves two simple poles $\lambda=\pm Z$ and should be understood using the limit absorption principle, i.e.,
\begin{equation}\label{dzGreensfunimplossless}
	G_{Z_{0}}(\bs r, \bs r')=\lim\limits_{\varepsilon\rightarrow 0}\frac{1}{2\pi}\int_{-\infty}^{+\infty}\frac{e^{\ri \lambda(x-x')-|\lambda|(y+y')}}{|\lambda|-Z_{\varepsilon}}d\lambda.
\end{equation}
% By using Cauchy principle, the limit can be removed and gives
% \begin{equation}\label{dzGreensfunimplossless} 
	% G_{Z_{0}}(\bs r, \bs r')=\frac{1}{2\pi}\int_{C_{\delta}}\frac{e^{\ri \lambda(x-x')-|\lambda|(y+y')}}{|\lambda|-Z_0}d\lambda,
	% \end{equation}
% where the contour has been changed to $C_{\delta}$, see Fig. \ref{newcontour} for an illustration. The parameter $\delta$ can be any given small positive number. 
% contour (see Fig. \ref{newcontour})
% \begin{equation}
	%    L_{\varepsilon}=\begin{cases}
		%        [0,+\infty),\quad \varepsilon>0,\\
		%        C_{\delta}^+:=C_{\delta}|_{\mathfrak{Re}(\lambda)\ge0},\quad \varepsilon=0,
		%    \end{cases} 
	% \end{equation}
%and 
Define integrals
\begin{equation}\label{generalintegral}
	\mathcal I_{n}(x, y)=\frac{1}{2\pi n!}\int_{0}^{+\infty}\frac{e^{-\lambda y+\ri\lambda x}\lambda^{n}}{\lambda-Z_{\varepsilon}}d\lambda,\quad n=0, 1, \cdots.
\end{equation}
% \begin{figure}[h!]
	% 	\centering
	% 	\includegraphics[width=0.6\textwidth]{pics/EIpath.pdf}
	% 	\caption{Integration path $C_{\delta}^+$.}
	% 	\label{newcontour}
	% \end{figure}
It is able to avoid the absolute value of $\lambda$ in the definition of $G_{Z_{\varepsilon}}(\bs r,\bs r')$ by rewriting it as
\begin{equation}\label{reactdecomposition}
	G_{Z_{\varepsilon}}(\bs r,\bs r')=\mathcal I_0(x-x', y+y')+\mathcal I_0(x'-x,  y+y').
\end{equation}

Hein et al. in \cite{Hein1} derived the following explicit expressions
\begin{align}
	G_{Z_{\varepsilon}}(\bs r, \bs r')=&-\frac{e^{-Z_{\varepsilon}(y+y')}}{2\pi}\left\{
	e^{\ri Z_{\varepsilon}(x-x')}{\rm{Ei}}\Big(Z_{\varepsilon}\big((y+y')-\ri(x-x')\big)\Big)\right.\nonumber\\
	&\left.+e^{-\ri Z_{\varepsilon}(x-x')}{\rm{Ei}}\Big(Z_{\varepsilon}\big((y+y')+\ri(x-x')\big)\Big)	
	\right\}, \quad \varepsilon>0,\nonumber
\end{align}	
and
\begin{align}
	G_{Z_{0}}(\bs r, \bs r')=&\lim\limits_{\varepsilon\rightarrow 0}G_{Z_{\varepsilon}}(\bs r, \bs r')
	=-\frac{e^{-Z(y+y')}}{2\pi}\left\{
	e^{\ri Z(x-x')}{\rm{Ei}}\Big(Z\big((y+y')-\ri(x-x')\big)\Big)\right.\nonumber\\
	&\left.+e^{-\ri Z(x-x')}{\rm{Ei}}\Big(Z\big((y+y')+\ri(x-x')\big)\Big)	\right\}\nonumber\\
	&+\ri e^{-Z(y+y')}\cos(Z(x-x')),\nonumber
\end{align}	
where 
\begin{align}\label{aaaaa1.1}
	{\rm {Ei}}(z)=-\int_{-z}^{\infty} \frac{e^{-t}}{t} dt, \quad z\ne0
\end{align}	
is the exponential integral function \cite{abramowitz1966handbook}. The exponential integral function ${\rm {Ei}}(z)$ with complex argument $z$ can use any contour form $-z$ to $\infty$ which does not cross the negative real axis or pass through the origin. These explicit expressions are useful to analyze the behavior of $G_{Z_{\varepsilon}}(\bs r,\bs r')$. However, it is complicate to establish far-field approximation theory from them. We will directly work on the integral form \eqref{dzGreensfunimp} and \eqref{dzGreensfunimplossless}. 

%The far-field behavior of $G_{Z_{\varepsilon}}(\bs r, \bs r')$ presented in \cite{Hein1} can be verified by the theory presented in the rest of this paper. 

% The far field of the complete Green's function \eqref{ddGreensfunimp}, denoted by $G^{f}$, describes its asymptotic behavior at infinity, i.e., when $|\bs r|\rightarrow\infty$ and assuming that $\bs r'$ is fixed.
% Using the similar analysis method as in reference \cite{Hein1}, we can obtain the far filed of the Green's function \eqref{ddGreensfunimp} as follows
% \begin{align}\label{fff1.0}
	% 	G^f =\begin{cases}
		% 	 -\frac{\sin{\theta}}{Z_{\varepsilon}\pi |\bs r|}(1-Z_{\varepsilon}y'),\quad \text{if}\ Z_{\varepsilon}=Z+\ri \varepsilon,\\[1em]
		% 	 -\frac{\sin{\theta}}{Z \pi |\bs r|}(1-Z y')+\ri e^{-Z(y+y')}\mathfrak{Re}\{e^{\ri Z|x|-\ri Zx'\sign{x}}\},\quad \text{if}\ Z_{\varepsilon}=Z,\\
		% 	\end{cases}
	% \end{align}
% where $\sin \theta=\frac{y}{|\bs r|}$. It is worth noting that the result of the first line of  
% equation \eqref{fff1.0} corresponds to the far field of the Green function with dissipation, whereas the result of the second line of equation \eqref{fff1.0} corresponds to the far field of the Green function without dissipation. Additionally, the previously mentioned outgoing radiation conditions \eqref{1.1} and \eqref{c1.1}-\eqref{c1.2} can be readily deduced by analyzing the properties of the far field \eqref{fff1.0}.

\subsection{Fast multipole method}\label{subsection2.11}
Let $\{(Q_{j},\boldsymbol{r}_{j}),$ $j=1,2,\cdots
,N\}$ be a large number of charged particles in the half plane, where $Q_j$ and $\boldsymbol{r}_{j}\in\mathbb R_+^2$ are the charge and coordinates of the $j$-th particle respectively. The potential of the interaction at any points $\bs r_i$ is given by the summation
\begin{equation}\label{potential1}
	\Phi(\boldsymbol{r}_{i})=\sum\limits_{j=1}^{N}Q_{j}G(\bs r_{i},\bs r_{j})
	:=\Phi^{\rm free}(\boldsymbol{r}_{i})+\Phi_1(\boldsymbol{r}_{i})+\Phi_2(\boldsymbol{r}_{i}),
\end{equation}
where
\begin{equation}\label{1.2}
	\Phi^{\rm free}(\boldsymbol{r}_{i})=-\frac{1}{2\pi}\sum\limits_{j=1,j\neq i}^{N}Q_{j}\ln(|\bs r_{ i}-\bs r_{j}|),
\end{equation}
is the free space component and 
\begin{equation}\label{reactioncomp1}	
	\Phi_1(\boldsymbol{r}_{i})=\frac{1}{2\pi}\sum\limits_{j=1}^{N}Q_{j}\ln(|{\bs r}_{ i}-\bs r_{j}^{\rm im}|),\quad \Phi_2(\boldsymbol{r}_{i})=\sum\limits_{j=1}^{N}Q_{j}G_{Z_{\varepsilon}}(\bs r_{ i},\bs r_{j}),
\end{equation}
are the reaction field components. The image sources are given by 
\begin{equation}
	\bs r_j^{\rm im}=(x_j, -y_j), \quad j=1, 2, \cdots, N.
\end{equation}
%$\bs r_j^{\rm im}=(x_j, -y_j)$ for all $j=1, 2, \cdots, N$. 
The classic FMM can be applied to compute $\{\Phi^{\rm free}(\boldsymbol{r}_{i})\}_{i=1}^N$ and $\{\Phi_1(\boldsymbol{r}_{i})\}_{i=1}^N$,  efficiently. Thus, we will first develop a FMM for efficient computation of the reaction component $\{\Phi_2(\boldsymbol{r}_{i})\}_{i=1}^N$. Then, a fast algorithm for the computation of the total potential $\{\Phi(\boldsymbol{r}_{i})\}_{i=1}^N$ can be made.

% The integral $\mathcal I_n(x, y)$ should also be understood as along the deformed integral $C_{\delta}$ if $\varepsilon=0$.
By the symmetry \eqref{greensfunsymm} and \eqref{reactdecomposition}, the reaction potential $\Phi_2(\boldsymbol{r}_{i})$ can be decomposed into the summation of
\begin{equation}\label{reactioncomp2}
	\Phi_2^{\pm}(\boldsymbol{r}_{i})=\sum\limits_{j=1}^{N}Q_{j}\mathcal I_0(\pm (x_{i}-x_{j}), y_{i}- y_{j}^{\rm im}),
\end{equation}
with $y_{j}^{\rm im}=-y_j$  namely,
\begin{equation}
	\Phi_2(\boldsymbol{r}_{i})=\Phi_2^{+}(\boldsymbol{r}_{i})+\Phi_2^{-}(\boldsymbol{r}_{i}).
\end{equation}
Therefore, we only need to focus on the FMM for $\Phi_2^{+}(\bs r_i)$, since $\Phi_2^{-}(\bs r_i)$ can be calculated similarly as
% since $\Phi_2^{-}(\bs r_i)$ can be calculated simply by replacing the target and source coordinates $(x_i, y_i)$, $(x_j, y_j)$ by $(-x_i, y_i)$, $(-x_j, y_j)$,  respectively,
in the algorithm for $\Phi_2^+(\bs r_i)$.

It is well known that the mathematical foundation of the FMM is the theory of the multipole and local expansions together with their shift and translation operators. Next, we will present these formulas for the reaction components $\Phi_2^{+}(\bs r_i)$. The key ingredient to derive the expansion theory is the following theorem whose proof will be given in Appendix A.
\begin{theorem}\label{proposition1}
	Given two points $\bs r=(x, y)\in\mathbb R^2_+,\ \bs r'=(x', y')\in\mathbb R^2_-$, $Z_{\varepsilon}=Z+\ri \varepsilon$ where $Z>0$, $\varepsilon > 0$, then
	\begin{equation}\label{reactfieldME}
		\mathcal I_0(x-x', y-y')=\sum\limits_{n=0}^{\infty}\ri^{-n}(x'+\ri y')^n\mathcal I_{n}(x,y),  
	\end{equation}
	holds for $|\bs r|>|\bs r'|$, and
	\begin{equation}\label{reactfieldLE}
		\mathcal I_0(x-x', y-y')=\sum\limits_{n=0}^{\infty}\ri^n(x+\ri y)^n\mathcal I_{n}(-x', -y'),
	\end{equation}
	holds for $|\bs r|<|\bs r'|$.
\end{theorem}

\begin{remark}
	From the proof of Theorem \ref{proposition1}, we can see that the convergence of the series is uniform w.r.t $\varepsilon>0$. Therefore, the expansions in Theorem \ref{proposition1} also hold for non-dissipative case as we have mentioned before.
\end{remark}

\begin{remark}
	The expansion theory derived in this section is a natural extension of that for free space Green's function. Actually, the reaction field of the half space problem with Dirichlet boundary condition also has integral representation
	\begin{equation}\label{halfspacereactfield}
		-\frac{1}{2\pi}\ln(|\bs r-\tilde{\bs r}'|)=2\mathfrak{Re}\Big[\frac{1}{4\pi}\int_{0}^{\infty}\frac{e^{-\lambda(y+y')}}{\lambda}e^{\ri \lambda (x-x')}d\lambda\Big],
	\end{equation}
	where $\bs r, \bs r'\in \mathbb R^2_+$ and $\tilde{\bs r}'=(x', -y')$ is the image coordinates of $\bs r'$ with respect to $y=0$. By the multipole expansion of $\ln(|\bs r-\tilde{\bs r}'|)$, we have
	\begin{equation}\label{logexpansion}
		-\frac{1}{2\pi}\ln(|\bs r-\tilde{\bs r}'|)=-\frac{1}{2\pi}\mathfrak{Re}\Big[\ln z-\sum\limits_{n=1}^{\infty}\frac{1}{n}\Big(\frac{\tilde{z}'}{z}\Big)^n\Big],
	\end{equation}
	where $z=x+\ri y$, $\tilde{z}'=x'-\ri y'$ are complex numbers corresponding to the coordinates. 
	On the other hand, applying \eqref{reactfieldME} to the integral representation \eqref{halfspacereactfield}, we obtain
	\begin{equation}\label{integralexpansion}
		-\frac{1}{2\pi}\ln(|\bs r-\tilde{\bs r}'|)
		=\mathfrak{Re}\Big[\sum\limits_{n=0}^{\infty}\frac{\tilde{z}'^n}{2\pi \ri^n n!}\int_{0}^{\infty}e^{-\lambda y+\ri \lambda x}\lambda^{n-1}d\lambda\Big].
	\end{equation}
	Note that
	\begin{equation}
		2\mathfrak{Re}\Big[\frac{1}{4\pi}\int_{0}^{\infty}\frac{e^{-\lambda y+\ri \lambda x}}{\lambda}d\lambda\Big]=\frac{1}{4\pi}\int_{-\infty}^{\infty}\frac{e^{-|\lambda| y+\ri \lambda x}}{|\lambda|}d\lambda=-\frac{1}{2\pi}\mathfrak{Re}\ln z.
	\end{equation}
	Moreover, by Cauchy theorem and the definition of Gamma function, we have
	\begin{equation}\label{aaa1.2}
		\int_{0}^{\infty}e^{-\lambda y+\ri \lambda x}\lambda^{n-1}d\lambda=\frac{1}{(y-\ri x)^n}\int_{C}e^{-z}z^{n-1}dz=\frac{\ri^n\Gamma(n)}{z^n},\quad n\geq 1,
	\end{equation}
	where the contour is defined as $C:=\{z=(y-\ri x)\lambda| \lambda\in[0, \infty)\}$. Substituting the above two identities into \eqref{integralexpansion}, we obtain exactly the classic expansion \eqref{logexpansion}. 
\end{remark}
\begin{figure}[h!]
	\centering
	\includegraphics[width=0.61\textwidth]{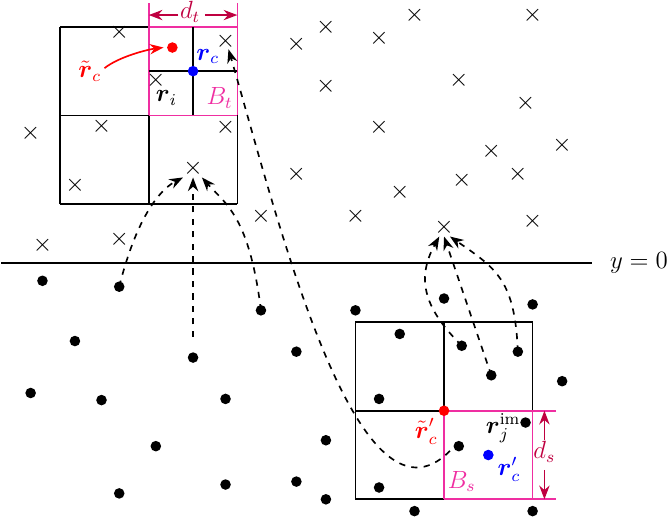}
	\caption{Targets (``$\times$"), images of the sources (``$\bullet$") and boxes in the FMM tree structure.}
	\label{field}
\end{figure} 

All far field approximations and their shifting and translation operators used in the FMM will be derived using the expansions in the Theorem \ref{proposition1}. The formulas used in the FMM for $\{\Phi_2^+(\bs r_i)\}_{i=1}^N$ are derived as follows:
\medskip
\begin{itemize}
	\item {\bf Source-to-Multipole (S2M):} Let $B_s$ be a box of size $d_s$ in $\mathbb R^2_-$ centered at $\bs r_c'=(x_c', y_c')$, $\bs r_i$ is any target coordinates in $\mathbb R^2_+$ such that $|\bs r_i-\bs r_c'|>\frac{\sqrt{2}}{2}d_s$, see Fig. \ref{field} for an illustration. 
	We consider the multipole expansion of the potential due to image sources inside $B_s$ at points $\bs r_j^{\rm im}$, i.e.,
	\begin{equation}\label{localpotential}
		\Phi_{2,B_s}^+(\bs r_i):=\sum\limits_{\bs r_j^{\rm im}\in B_s}Q_{j}\mathcal I_0(x_{i}-x_{j}, y_{i}- y_{ j}^{\rm im}).
	\end{equation}
	We firt add the center $\bs r_c'$ into the integral as follows
	\begin{equation}
		\mathcal I_0(x_i-x_j, y_i-y_j^{\rm im})=\mathcal I_0(x_c'-x_j-(x_c'-x_i), y_c'-y_j^{\rm im}-(y_c'-y_i)),
	\end{equation}
	and then apply the expansion \eqref{reactfieldLE} which gives us
	\begin{equation}\label{MEformula}
		\Phi_{2,B_s}^+(\bs r_i)=\sum\limits_{n=0}^{\infty}\alpha_{n}\mathcal I_n(x_i-x_c', y_i-y_c'),
	\end{equation}
	where
	\begin{equation}\label{MEcoefficients}
		\begin{split}
			\alpha_{n}=\sum\limits_{\bs r_j^{\rm im}\in B_s}Q_j\ri^{n}[(x_c'-x_j)+\ri (y_c'-y_j^{\rm im})]^n.
		\end{split}
	\end{equation}
	In the FMM, the truncated expansion
	\begin{equation}\label{truncatedME}
		\Phi_{ME}^{p}(\bs r_i)=\sum\limits_{n=0}^{p}\alpha_{n}\mathcal I_n(x_i-x_c', y_i-y_c'),
	\end{equation}
	is used as far field approximation for $\Phi_{2,B_s}^+(\bs r_i)$ where $p$ is the truncation order. 
	%We will prove that $\Phi_{ME}^{N}(\bs r_i^{\rm im})$ exponentially converges to $\Phi_{2,B_s}^+(\bs r_i^{\rm im})$ as $N\rightarrow\infty$.
	\medskip
	
	\item {\bf Multipole-to-Multipole (M2M):} Let $\tilde{\bs r}_c'=(\tilde x_c', \tilde y_c')$ be the center of the parent box of $B_s$ and further assume $\bs r_i$ satisfies $|\bs r_i-\tilde{\bs r}_c'|>\max\limits_{\bs r_j^{\rm im}\in B_s} |\bs r_j^{\rm im}-\tilde{\bs r}_c'|$. Apparently, we have multipole expansion
	\begin{equation}\label{M2M0}
		\Phi_{2,B_s}^+(\bs r_i)=\sum\limits_{n=0}^{\infty}\tilde \alpha_{n}\mathcal I_n(x_i-\tilde x'_c, y_i-\tilde y_c'),
	\end{equation}
	with respect to the center $\tilde{\bs r}_c'$ , where
	\begin{equation}
		\tilde\alpha_{n}=\sum\limits_{\bs r_j^{\rm im}\in B_s}Q_j\ri^{n}[(\tilde x'_c-x_j)+\ri(\tilde y'_c-y_j^{\rm im})]^n.
	\end{equation}
	By binomial formula and \eqref{MEcoefficients}, we have
	\begin{equation}\label{metome}
		\begin{split}
			\tilde\alpha_{n}=&\sum\limits_{\bs r_j^{\rm im}\in B_s}Q_j\ri^{n}[(\tilde x'_c-x_c'+x_c'-x_j)+\ri(\tilde y_c'-y_c'+y_c'-y_j^{\rm im})]^n\\
			=&\sum\limits_{m=0}^n\sum\limits_{\bs r_j^{\rm im}\in B_s}Q_j\frac{\ri^nn![(x'_c-x_j)+\ri (y_c'-y_j^{\rm im})]^m}{(n-m)!m![(\tilde x'_c-x_c')+\ri(\tilde y'_c-y_c')]^{m-n}}\\
			=&\sum\limits_{m=0}^n\frac{n!\ri^{n-m}}{(n-m)!m!}[(\tilde x'_c-x_c')+\ri(\tilde y'_c-y_c')]^{n-m}\alpha_m.
		\end{split}
	\end{equation}
	This is exactly the M2M shifting operator used in the classic FMM for free space problems. With the truncated ME \eqref{truncatedME}, we can exactly calculate the truncated version of \eqref{M2M0} which we denoted by
	\begin{equation}\label{truncatedM2M}
		\widetilde{\Phi}_{ME}^{p}(\bs r_i)=\sum\limits_{n=0}^{p}\tilde{\alpha}_{n}\mathcal I_n(x_i-\tilde x_c', y_i-\tilde y_c').
	\end{equation}
	%As the truncated ME $\Phi_{ME}^{N}(\bs r_i^{\rm im})$, it will also converge to $\Phi_{2,B_s}^+(\bs r_i^{\rm im})$ exponentially as $N\rightarrow\infty$.
	
	\item {\bf Source-to-Local (S2L):} Let $B_t$ be a target box of size $d_t$ in $\mathbb R^2_+$ centered at $\bs r_c=(x_c, y_c)$. Suppose all image sources $\bs r_j^{\rm im}$ in $B_s$ satisfy $\min\limits_{\bs r_j^{\rm im}\in B_s}|\bs r_j^{\rm im}-{\bs r}_c|>\frac{\sqrt{2}}{2}d_t$. Then, the potential due to the image sources inside $B_s$ at any points $\bs r_i\in B_t$ has expansion 
	\begin{equation}\label{LEformula}
		\Phi_{2,B_s}^+(\bs r_i)=\sum\limits_{n=0}^{\infty}\beta_{n}[(x_c-x_i)+\ri (y_c-y_i)]^n,
	\end{equation}
	where
	\begin{equation}\label{LEcoefficients}
		\beta_{n}=\sum\limits_{\bs r_j^{\rm im}\in B_s}Q_j\ri^{-n}\mathcal I_n(x_c-x_j, y_c-y_j^{\rm im}).
	\end{equation}
	The truncated expansion
	\begin{equation}\label{truncatedLE}
		\Phi_{LE}^{p}(\bs r_i)=\sum\limits_{n=0}^{p}\beta_{n}[(x_c-x_i)+\ri (y_c-y_i)]^n,
	\end{equation}
	is the so called local expansion for $\Phi_{2,B_s}^+(\bs r_i)$ used in the FMM. 
	%Similaly, we will prove that $\Phi_{LE}^{N}(\bs r_i^{\rm im})$ exponentially converges to $\Phi_{2,B_s}^+(\bs r_i^{\rm im})$ as $N\rightarrow\infty$.
	\medskip
	
	\item {\bf Local-to-Local (L2L):} Further assume $\tilde{\bs r}_c=(\tilde x_c, \tilde y_c)$ be the center of a child box of $B_t$ and $\min\limits_{\bs r_j^{\rm im}\in B_s}|\bs r_j^{\rm im}-\tilde{\bs r}_c|>\frac{\sqrt{2}}{4}d_t$. By the local expansion in \eqref{LEformula} and binomial formula, we have
	\begin{equation}\label{L2L0}
		\begin{split}
			\Phi_{2,B_s}^+(\bs r_i)
			=&\sum\limits_{n=0}^{\infty}\beta_{n}[(x_c-\tilde x_c+\tilde x_c-x_i)+\ri(y_c -\tilde y_c+\tilde y_c-y_i)]^n\\		=&\sum\limits_{n=0}^{\infty}\beta_{n}\sum\limits_{m=0}^n\frac{n![(\tilde x_c-x_i)+\ri(\tilde y_c-y_i)]^m}{(n-m)!m![(x_c-\tilde x_c)+\ri ( y_c-\tilde y_c)]^{m-n}}\\
			=&\sum\limits_{m=0}^{\infty}\sum\limits_{n=m}^{\infty}\beta_{n}\frac{n![(\tilde x_c-x_i)+\ri(\tilde y_c-y_i)]^m}{(n-m)!m![(x_c-\tilde  x_c)+\ri (y_c-\tilde  y_c)]^{m-n}}\\
			=&\sum\limits_{m=0}^{\infty}\tilde\beta_m[(\tilde x_c-x_i)+\ri(\tilde y_c-y_i)]^m,
		\end{split}
	\end{equation}
	where 
	\begin{equation}\label{letole1}
		\tilde\beta_m=\sum\limits_{n=m}^{\infty}\frac{n![( x_c-\tilde x_c)+\ri(y_c-\tilde y_c) ]^{n-m}}{(n-m)!m!}\beta_{n}.
	\end{equation}
	This is again the exact L2L formulation used in the classic FMM for free space problems. In the implementation of the FMM, given a truncated LE \eqref{truncatedLE}, \eqref{letole1} has to be truncated to $n=p$. Then, the L2L shifting produce the following approximation
	% \begin{equation}\label{truncatedL2L}
		% \widetilde{\Phi}_{LE}^{p}(\bs r_i^{\rm im}) =\sum\limits_{n=0}^{p}\breve{\beta}_{n}[(x_c-x_i)+\ri (y_c-y_i^{\rm im})]^n,
		% \end{equation}
	% where 
	% $$\breve{\beta}_{n}=\sum\limits_{n=m}^{p}\frac{n![( x_c-\tilde x_c)+\ri(y_c-\tilde y_c) ]^{n-m}}{(n-m)!m!}\beta_{n}.$$
	\begin{equation}\label{truncatedL2L}
		\widetilde{\Phi}_{LE}^p(\bs r_i) =\sum\limits_{m=0}^{p}\breve\beta_m[(\tilde x_c-x_i)+\ri(\tilde y_c-y_i)]^m,
	\end{equation}
	where 
	$$\breve\beta_m=\sum\limits_{n=m}^{p}\frac{n![( x_c-\tilde x_c)+\ri(y_c-\tilde y_c) ]^{n-m}}{(n-m)!m!}\beta_{n}.$$

	\item {\bf Multipole-to-Local (M2L):} Suppose the aforementioned source and target boxes $B_s$ and $B_t$ satisfy $|\bs r_c-\bs r_c'|>\frac{\sqrt{2}(d_s+d_t)}{2}$, where $d_s, d_t$ are the size of $B_s$ and $B_t$, respectively. Then, for any $\bs r_i$ in $B_t$, we have  multipole expansion
	\begin{equation*}
		\Phi_{2,B_s}^+(\bs r_i)=\sum\limits_{n=0}^{\infty}\alpha_{n}\mathcal I_n(x_i-x_c+x_c-x'_c, y_i-y_c+y_c-y_c'),
	\end{equation*}
	where $\{\alpha_n\}_{n=0}^{\infty}$ are given by \eqref{MEcoefficients}. Applying expansion \eqref{reactfieldME} again, we obtain
	\begin{equation*}
		\Phi_{2,B_s}^+(\bs r_i) =\sum\limits_{m=0}^{\infty}\sum\limits_{n=0}^{\infty}\frac{\alpha_{n}\ri^{-m}(n+m)!\mathcal I_{n+m}( x_c-x_c',y_c-y_c')}{n!m![(x_c-x_i)+\ri(y_c-y_i)]^{-m}}.
	\end{equation*}
	Comparing with the local expansion \eqref{LEformula} gives
	\begin{equation}\label{metoleimage}
		\beta_m=\sum\limits_{n=0}^{\infty}\frac{\ri^{-m}(n+m)!}{ n!m!}\mathcal I_{n+m}(x_c- x'_c,y_c-y'_c)\alpha_{n}.
	\end{equation}
	Similar as in the L2L shifting, given a truncated LE \eqref{truncatedLE}, \eqref{metoleimage} has to be truncated to $n=p$. Then, the M2L translation produce the following approximation
	\begin{equation}\label{truncatedM2L}
		{\Phi}_{M2L}^{p}(\bs r_i) =\sum\limits_{n=0}^{p}\hat{\beta}_{n}[(x_c-x_i)+\ri (y_c-y_i)]^n,
	\end{equation}
	where 
	\begin{equation}\label{truncatedM2Lcoef}
		\hat{\beta}_{n}=\sum\limits_{m=0}^{p}\frac{\ri^{-n}(n+m)!}{ n!m!}\mathcal I_{n+m}(x_c- x'_c,y_c-y'_c)\alpha_{m}.
	\end{equation}
\end{itemize}

Using the truncated expansions, shifting and translation operators in the framework of the classic FMM, we implement an FMM for fast calculation of $\{\Phi_2^+(\bs r_i)\}_{i=1}^N$ at any desired accuracy. The pseudo-code of the algorithm is presented in Algorithm \ref{algorithm1} and the overall FMM for the computation of the interactions \eqref{potential1} is presented in Algorithm \ref{algorithm2}.
\begin{figure}[h!]
	\centering
	\includegraphics[width=0.6\textwidth]{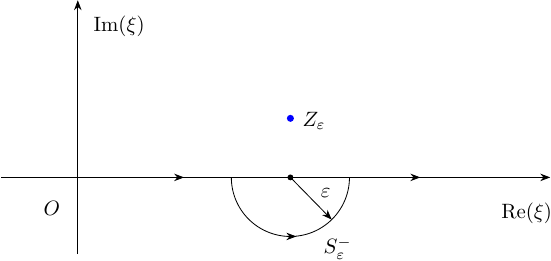}
	\caption{The integration path $L_{\varepsilon}$ for lossless scenario.}
	\label{newcontour}
\end{figure}

The FMM for the reaction components requires an efficient algorithm to compute the integrals $\mathcal I_n(x,y)$. By
\begin{equation*}
	\frac{\lambda^n}{\lambda-z}=\lambda^{n-1}+z\frac{\lambda^{n-1}}{\lambda-z},\quad n=1, 2, \cdots,
\end{equation*} 
we have recursion
\begin{equation}\label{integralrecursion}
	\begin{split}
		\mathcal I_n(x,y)=&
		\frac{1}{2\pi n!}\int_{0}^{\infty}e^{-\lambda y+\ri \lambda x}\left(\lambda^{n-1}+Z_{\varepsilon}\frac{\lambda^{n-1}}{\lambda-Z_{\varepsilon}}\right)d\lambda\\
		=&\frac{1}{2\pi n(y-\ri x)^{n}}+\frac{Z_{\varepsilon}}{n}\mathcal I_{n-1}(x, y).
	\end{split}
\end{equation}
The initial value $\mathcal I_0(x,y)=-\frac{1}{2\pi}e^{-Z_{\varepsilon}\left( y+\ri x \right)}{\rm{Ei}}\left( Z_{\varepsilon}\left( y+\ri x \right) \right)$ where ${\rm{Ei}}(z)$ is the exponential integral defined in \eqref{aaaaa1.1}. Efficient implementations for the computation of  ${\rm{Ei}}(z)$ can be obtained in many well-known packages. Therefore, the translation operator \eqref{metoleimage} can be calculated very efficiently using the recurrence formula \eqref{integralrecursion}.
\begin{algorithm}[H]
	\caption[Algorithm 1]{The FMM for reaction component $\Phi_2^{+}(\boldsymbol{r}_{i}), i=1, 2, \cdots, N$}
	\label{algorithm1}
	\begin{algorithmic}
		\State{Generate image coordinates $\bs r_j^{\rm im}$ for all sources $\bs r_j$.}
		\State{Generate an adaptive hierarchical tree structure with target points $\{\bs r_i\}_{i=1}^N$ and image source points $\{\bs r_j^{\rm im}\}_{j=1}^N$ and pre-compute some tables.}
		\State{\bf Upward pass:}
		\For{$\ell=H \to 0$}
		\For{ all boxes $j$ on source tree level $\ell$}
		\If{$j$ is a leaf node}
		\State{form ME using Eq. \eqref{MEcoefficients}.}
		\Else
		\State{form ME by merging children's MEs using shifting operator \eqref{metome}.} 
		\EndIf
		\EndFor
		\EndFor
		\State{\bf Downward pass:}
		\For{$\ell=1 \to H$}
		\For{all boxes $j$ on target tree level $\ell$}
		\State{shift the LE of $j$'s parent to $j$ itself using shifting operator \eqref{letole1}.}
		\State{collect interaction list contribution using M2L translation operator \eqref{metoleimage}.}
		\EndFor
		\EndFor
		\State {\bf Evaluate Local expansions:}
		\For{each leaf node (childless box)}
		\State{evaluate the local expansion at each particle location using \eqref{LEformula}.} 
		\EndFor
		\State {\bf Local Direct Interactions:}
		\For{$i=1 \to N$ }
		\State{compute \eqref{reactioncomp2} of target particle $i$ in the neighboring boxes.}
		\EndFor
	\end{algorithmic}
\end{algorithm}
\begin{algorithm}[H]
	\caption[Algorithm2]{The overall FMM for the interactions $\Phi(\bs r_{i}),i=1, 2, \cdots,N$}
	\label{algorithm2}
	\begin{algorithmic}
		\State compute $\{\Phi^{\rm free}(\bs r_{i})\}_{i=1}^N$  using free space FMM.
		\State compute $\{\Phi_1(\bs r_{i})\}_{i=1}^N$ using free space FMM with image sources at $\{\bs r_j^{\rm im}\}_{j=1}^N$ and targets at $\{{\bs r}_i\}_{i=1}^N$.
		\State compute $\{\Phi_2^+(\bs r_{i})\}_{i=1}^N$  using Algorithm \ref{algorithm1} with image sources at $\{\bs r_j^{\rm im}\}_{j=1}^N$ and targets at $\{\bs r_{i}\}_{i=1}^N$.
		\State compute $\{\Phi_2^-(\bs r_{i})\}_{i=1}^N$ using Algorithm \ref{algorithm1} with image sources at $\{(-x_j, y_j^{\rm im})\}_{j=1}^N$ and targets at $\{(-x_i, y_i)\}_{i=1}^N$.
		\State sum all the components together to obtain $\{\Phi(\bs r_i)\}_{i=1}^N$.
	\end{algorithmic}
\end{algorithm}

For the scenario without dissipation, the limit absorption principle should be used when taking the limit $\varepsilon\rightarrow 0$ on both sides of \eqref{integralrecursion}, i.e., the path for the integrals should be deformed to $L_{\varepsilon}$ as depicted in Fig. \ref{newcontour}. Therefore, we still have the recurrence formula \eqref{integralrecursion} for the lossless case while the initial value is given by
\begin{equation*}  \label{aaa1.5}
	\begin{split}
		\lim\limits_{\varepsilon\rightarrow 0}{\rm{Ei}}(Z_{\varepsilon}(y+\ri x))
		=&-e^{Z(y+\ri x)}\lim\limits_{\varepsilon\rightarrow 0}\int_{L_{\varepsilon}}\frac{e^{-\xi(y+\ri x)}}{\xi-Z_{\varepsilon}}d\xi\\
		=&-e^{Z(y+\ri x)}\lim\limits_{\varepsilon\rightarrow 0}\left[\left(\int_{0}^{Z-\varepsilon}+\int_{Z+\varepsilon}^{\infty}+\int_{S^-_{\varepsilon}}\right)\frac{e^{-\xi(y+\ri x)}}{\xi-Z_{\varepsilon}}d\xi \right]\\
		=&-e^{Z(y+\ri x)}\left[ {\rm p.v.}\int_{0}^{\infty}\frac{e^{-\xi(y+\ri x)}}{\xi-Z}d\xi+\ri \pi e^{-Z(y+\ri x)} \right]\\
		=&-e^{Z(y+\ri x)}\left[-e^{-Z(y+\ri x)}{\rm{Ei}}(Z(y+\ri x))+\ri \pi e^{-Z(y+\ri x)}\right].
	\end{split}
\end{equation*}

\section{Error analysis for the FMM}\label{section4}
In this section, an error estimate of the FMM for the reaction component $\{\Phi_2^{\pm}(\bs r_i)\}_{i=1}^N$ is established. We prove that the FMM for the reaction component enjoys a similar exponential convergence as the classic FMM for free space case, except the convergence rates are determined by the Euclidean distance between the image and source points. 
Before estimating the error, it is essential to state the following key estimate.
\begin{lemma}\label{lemma1}
	Given any complex number $z=a+{\rm i} b$, we have
	\begin{equation}\label{aaaaaa1.0}
		\left|\int_0^{+\infty}\frac{e^{-t}t^{n}}{t-z}dt\right|\le C\max\{|z|^2, 1\}(n-1)!
	\end{equation}
	for $n\geq |z|e$, where $C$ is a constant independent of $z$ and $n$. The contour is set to be $\Gamma_1\cup \Gamma_3$ as presented in \eqref{f1.12} to get rid of the pole $t=z$ in the case $a>0$, $b=0$. 
\end{lemma}

This lemma can be thought of as an extension of the identity
$$(n-1)!=\Gamma(n):=\int_0^{\infty}e^{-t}t^{n-1}dt,$$
and the proof will be given in Appendix B. 

\begin{lemma}\label{theorem1}
	Given a point $\bs r=(x,y)\in\mathbb{R}_+^2$ and complex number $Z_{\varepsilon}=Z+\ri\varepsilon$ in the first quadrant, i.e., $Z>0$, $\varepsilon> 0$, there exists a positive constant $C$ such that for $n\in \mathbb{N}$, $n\geq e|Z_{\varepsilon}||\bs r|$, we have
	\begin{equation}\label{aa1.0}
		|\mathcal I_n(x, y)|\le \frac{C\max\{|Z_{\varepsilon}|^2|\bs r|^2, 1\}}{n|\bs r|^n}.
	\end{equation}
\end{lemma}
\begin{proof}
	Applying variable substitution $t=(y-\ri x)\lambda$, we have
	\begin{equation}
		\mathcal I_n(x, y)=\frac{1}{2\pi n!}\int_{0}^{+\infty}\frac{e^{-\lambda y+\ri \lambda x}\lambda^n}{\lambda-Z_{\varepsilon}}d\lambda=\frac{1}{2\pi n!(y-\ri x)^n}\int_{S}\frac{t^{n}e^{-t}}{t-Z_{\varepsilon}(y-\ri x)}dt,
	\end{equation}
	where $S=\big\{t=(y-\ri x)\lambda|\lambda\in[0,+\infty) \big\}$, see Fig. \ref{f1.4} for a sketch.
	Then, the estimate can be obtained by changing the contour $S$ to the real line and then applying the estimate in Lemma \ref{lemma1}. The contour deformation depends on the position of $P=Z_{\varepsilon}(y-\ri x)$. 
	
	If $x=0$, the contour $S$ is the real line, i.e.,
	\begin{equation}
		\int_{S}\frac{e^{-t}t^{n}}{t-P}dt
		=\int_{0}^{+\infty}\frac{e^{-t}t^{n}}{t-Z_{\varepsilon}y}dt,
	\end{equation}
	where $P=Z_{\varepsilon}y$ locates in the first quadrant due to the assumption $Z>0, \varepsilon>0$ and $y>0$. 
	
	If $x<0$, the contour $S$ is in the first quadrant and the inequality
	\begin{equation}
		\frac{\mathfrak{Im}(P)}{\mathfrak{Re}(P)}=\frac{-x+\frac{\varepsilon y}{Z}}{y+\frac{\varepsilon x}{Z}}>\frac{-x}{y},\quad {\rm if}\;\;y+\frac{\varepsilon x}{Z}>0,
	\end{equation}
	shows that the point given by $P$ is located above the contour $S$ or in the left complex plane, see Fig. \ref{f1.4} (a). Therefore, we can apply Cauchy's theorem to change the contour from $S_R=\big\{t=(y-\ri x)\lambda:\lambda\in[0,R] \big\}$ to $\Gamma_{R,1}\cup \Gamma^+_{R,2}$, where
	\begin{align}\label{f1.3}
		\Gamma_{R,1}=\left\{t:0\le t \le R\right\},\quad
		\Gamma^+_{R,2}=\left\{t=Re^{\ri \theta} :0\le \theta \le \arctan \Big(-\frac{x}{y}\Big)\right\}\nonumber.
	\end{align}
	For the integral along $\Gamma^+_{R,2}$, the assumption $y>0$ implies that
	\begin{align}
		\lim\limits_{R\rightarrow+\infty}\Big|\int_{\Gamma^+_{R,2}}\frac{e^{-t}t^{n}}{t-P}dt\Big|
		=&\lim\limits_{R\rightarrow+\infty}\Big|\int_{0}^{\arctan \big(-\frac{x}{y} \big)}\frac{e^{-Re^{\ri \theta}}(Re^{\ri \theta})^nRe^{\ri \theta}\ri}{Re^{\ri \theta}-P}d\theta\Big|\\
		\leq&\lim\limits_{R\rightarrow+\infty}e^{-\frac{Ry}{\sqrt{x^2+y^2}}}R^{n}\int_{0}^{\arctan \big(-\frac{x}{y}\big)}\frac{1}{|e^{\ri\theta}-P/R|} d\theta=0.\nonumber
	\end{align}
	Therefore, the Cauchy theorem gives
	\begin{equation}\label{m1.2}
		\int_{S}\frac{e^{-t}t^{n}}{t-P}dt
		=\lim\limits_{R\rightarrow+\infty}\int_{\Gamma_{R,1}}\frac{e^{-t}t^{n}}{t-P}dt=\int_{0}^{+\infty}\frac{e^{-t}t^{n}}{t-P}dt,
	\end{equation}
	where $\mathfrak{Im}(P)=\varepsilon y-Zx>0$.
	
	If $x>0$, the contour $S$ is lies in the fourth quadrant and the imaginary part $\mathfrak{Im}(P)=\varepsilon y-Zx$ can be any real number. If $\mathfrak{Im}(P)>0$, the contour is illustrated in Fig. \ref{f1.4} (b), we can directly change the contour to the real line and obtain \eqref{m1.2} with $P$ located in the first quadrant. If $\mathfrak{Im}(P)\leq 0$, the contour changes are depicted in Fig. \ref{f1.4} (c) (d). 
	We can also check that, the integral along contour
	\begin{equation}\label{contournegative1}
		\Gamma_{R, 2}^-=\left\{t=Re^{\ri \theta} :\theta_{xy}:=\arctan \Big(-\frac{x}{y}\Big)\le \theta \le 0\right\}
	\end{equation}
	tends to $0$ as $R\rightarrow+\infty$. 
	Therefore, we can always have
	\begin{equation}\label{f1.11}
		\int_{S}\frac{e^{-t}t^{n}}{t-P}dt
		=\lim\limits_{R\rightarrow+\infty}\int_{\Gamma_{\delta,1}}\frac{e^{-t}t^{n}}{t-P}dt+\int_{\Gamma_{\delta}}\frac{e^{-t}t^{n}}{t-P}dt,
	\end{equation}
	where
	\begin{equation*}%\label{f1.12}
		\Gamma_{\delta,1}=[0, \mathfrak{Re}(P)-\delta]\cup [\mathfrak{Re}(P)+\delta, R],\quad
		\Gamma_{\delta}=\left\{t=\delta e^{\ri \theta}+\mathfrak{Re}(P) :-\pi\le \theta \le 0 \right\},
	\end{equation*}
	for $\mathfrak{Im}(P)=0$, or 
	\begin{equation}\label{f1.16}
		\begin{split}
			\Gamma_{\delta,1}=\left\{t:0\le t\le R\right\},\quad
			\Gamma_{\delta}=C_{\delta}=\left\{t=\delta e^{\ri \theta}+P :0\le \theta \le 2\pi \right\},
		\end{split}
	\end{equation}
	for $\mathfrak{Im}(P)<0$. 
	
	Apparently, in all cases discussed above, we can directly apply Lemma \ref{lemma1} to obtain desired estimate for the integrals after contour change. One extra estimate for the integral along the contour $\Gamma_3$ is required in the proof of the case $\mathfrak{Im}(P) < 0$. Actually, direct calculation gives 
	\begin{equation}\label{f1.18}
		\begin{split}
			\left|\int_{\Gamma_{\delta}}\frac{e^{-t}t^{n}}{t-P}dt\right|
			=&\left|\lim\limits_{\delta\rightarrow0^+}\int_{0}^{2\pi}e^{-(\delta e^{\ri \theta}+P)}(\delta e^{\ri \theta}+P)^{n}\ri d\theta\right|\\
			=&\left|2\pi\ri e^{-P}P^{n}\right|
			=2\pi |P|^{n}|e^{-P}|\le2\pi (n-1)!|P|.
		\end{split} 
	\end{equation}	
\end{proof}

In the adaptive FMM, the approximations \eqref{truncatedME},\eqref{truncatedLE},\eqref{truncatedM2M},\eqref{truncatedL2L} and \eqref{truncatedM2L} could be used separately or together to generate approximations for $\Phi_{2,B_s}^+(\bs r_i)$. Now, we will prove error estimates for these approximations one-by-one.

\begin{theorem}\label{theorem2}
	The ME given in \eqref{MEformula} possesses a truncation error estimate
	\begin{equation}\label{f1.8}
		\begin{split}
			&\left|\Phi_{2,B_s}^+(\bs r_i)-\sum\limits_{n=0}^{p}\alpha_{n}\mathcal I_n(x_i-x_c', y_i-y_c')\right|
			\le\frac{CC_{max}Q_{s}}{p+1}q^{p+1},
		\end{split}
	\end{equation}
	where 
	\begin{equation}\label{localcharge}
		Q_s=\sum\limits_{\bs r_j^{\rm im}\in B_s}|Q_j|,\quad q=\frac{\sqrt{2}d_s}{2|\bs r_i-\bs r'_{c}|}<1,\quad C_{max}=\frac{\max\{|Z_{\varepsilon}|^2|\bs r_i-\bs r_c'|^2, 1\}}{1-q},
	\end{equation}
	and $C$ is a positive generic constant.
\end{theorem}
\begin{proof}
	Using \eqref{generalintegral}, \eqref{MEformula}-\eqref{MEcoefficients} and Lemma \ref{theorem1}, together with the given conditions, we can obtain
	\begin{equation}\label{f1.1}
		\begin{split}
			&\left|\Phi_{2,B_s}^+(\bs r_i)-\sum\limits_{n=0}^{p}\alpha_{n}\mathcal I_n(x_i-x'_c, y_i-y_c')\right|\\
			=&\left|\sum\limits_{n=p+1}^{\infty}\left(\sum\limits_{\bs r_j^{\rm im}\in B_s}Q_j[(x'_c-x_j)+\ri (y_c'-y_j^{\rm im})]^n \right)\mathcal I_n(x_i-x'_c, y_i-y_c')\right|\\
			\leq&\sum\limits_{n=p+1}^{\infty}Q_s\Big(\frac{\sqrt{2}}{2}d_s\Big)^n \left|\mathcal I_n(x_i-x'_c, y_i-y_c')\right|\\
			\le&\sum\limits_{n=p+1}^{\infty}Q_s \left(\frac{\sqrt{2}d_s}{2|\bs r_i-\bs r_c'|}\right)^n\frac{C\max\{|Z_{\varepsilon}|^2|\bs r_i-\bs r_c'|^2, 1\}}{n}\\
			\le&\frac{CQ_s\max\{|Z_{\varepsilon}|^2|\bs r_i-\bs r_c'|^2, 1\}}{p+1}\frac{q^{p+1}}{1-q}.
		\end{split}
	\end{equation}
\end{proof}
\begin{theorem}\label{theorem3} 
	The LE given in \eqref{LEformula} possesses a truncation error estimate
	\begin{equation}\label{g1.2}
		\begin{split}
			&\left|\Phi_{2,B_s}^+(\bs r_i)-\sum\limits_{n=0}^{p}\beta_{n}[(x_c-x_i)+\ri (y_c-y_i)]^n\right|
			\le\frac{C\widetilde C_{max}Q_s}{p+1}\tilde{q}^{p+1},
		\end{split}
	\end{equation}
	where $Q_s$ is defined in \eqref{localcharge}, $$\tilde{q}=\frac{|\bs r_{i}-\bs r_{c}|}{\min\limits_{\bs r_j^{\rm im}\in B_s} |\bs r_j^{\rm im}-{\bs r}_c|}<1,\quad\widetilde C_{max}=\frac{\max\{|Z_{\varepsilon}|^2|\bs r_j^{\rm im}-\bs r_c|^2, 1\}}{1-\tilde q},$$
	and $C$ is a positive generic constant.
\end{theorem}
\begin{proof}
	The truncation error can be obtained similarly by applying Lemma \ref{theorem1} as follows
	\begin{equation}\label{g1.1}
		\begin{split}
			&\left|\Phi_{2,B_s}^+(\bs r_i)-\sum\limits_{n=0}^{p}\beta_{n}[(x_c-x_i)+\ri (y_c-y_i)]^n\right|\\	=&\left|\sum\limits_{n=p+1}^{\infty}\left(\sum\limits_{\bs r_j^{\rm im}\in B_s}Q_j\ri^{-n}\mathcal I_n(x_c-x_j, y_c-y_j^{\rm im})\right) [(x_c-x_i)+\ri (y_c-y_i)]^n\right|\\
			\leq&\sum\limits_{n=p+1}^{\infty}\sum\limits_{\bs r_j^{\rm im}\in B_s}|Q_j|\left|\mathcal I_n(x_c-x_j, y_c-y_j^{\rm im})
			\right| |\bs r_i-\bs r_c|^n\\
			\le&\sum\limits_{n=p+1}^{\infty}\Big(\frac{|\bs r_i-\bs r_c|}{\min\limits_{\bs r_j^{\rm im}\in B_s} |\bs r_j^{\rm im}-{\bs r}_c|}\Big)^n \frac{CQ_s\max\{|Z_{\varepsilon}|^2|\bs r_j^{\rm im}-\bs r_c|^2, 1\}}{n}\\
			\le&\frac{CQ_s\max\{|Z_{\varepsilon}|^2|\bs r_j^{\rm im}-\bs r_c|^2, 1\}}{p+1}\frac{\tilde{q}^{p+1}}{1-\tilde{q}}.
		\end{split}
	\end{equation} 
\end{proof}	

The M2M shifting formulation \eqref{metome} shows that ME coefficients $\tilde\alpha_n$ can be calculated exactly from $\{\alpha_k\}_{k=0}^n$. Therefore, no additional error is introduced during the M2M shifting and the total error due to S2M and then M2M is equal to the error we do ME approximation directly at the new center $\tilde r_c'$. Therefore, the ME with coefficients $\tilde\alpha_n$ calculated from M2M shifting formulation \eqref{metome} has the following error estimate. 
\begin{theorem}\label{theorem4}
	The ME shifting formulation \eqref{M2M0} has a truncation error estimate
	\begin{equation}\label{ff1.8}
		\begin{split}
			&\left|\Phi_{2,B_s}^+(\bs r_i)-\sum\limits_{n=0}^{p}\tilde \alpha_{n}\mathcal I_n(x_i-\tilde x'_c, y_i-\tilde y_c')\right|
			\le\frac{CQ_s}{p+1}\frac{\hat{q}^{p+1}}{1-\hat{q}},
		\end{split}
	\end{equation}
	where $\hat{q}=\frac{\max\limits_{\bs r_j^{\rm im}\in B_s} |\bs r_j^{\rm im}-\tilde{\bs r}_c'|}{|\bs r_{i}-\bs{\tilde r}'_{c}|}<1$ and $C$ is a positive constant.
\end{theorem}

Observing the coefficient conversion formula \eqref{letole1} from local to local (L2L) translation, we note that truncation introduces errors into the formula. Therefore, we need to provide an analysis of the truncation error for the local expansion of local to local (L2L) translation.
\begin{theorem}\label{theorem5}
	The LE shifting formulation \eqref{L2L0} has a truncation error estimate
	\begin{equation}\label{fff1.8}
		\left|\Phi_{2,B_s}^+(\bs r_i)-\sum\limits_{m=0}^{p}\breve\beta_{m}[(\tilde x_c-x_i)+\ri(\tilde y_c-y_i) ]^{m}\right|
		\le\frac{CQ_s}{p+1}\frac{\tilde{q}^{p+1}}{1-\tilde{q}},
	\end{equation}
	where $\tilde{q}=\frac{|\bs r_{c}-\bs r_{i}|}{\min\limits_{\bs r_j^{\rm im}\in B_s} |\bs r_c-{\bs r}_j^{\rm im}|}<1$ and $C$ is a positive constant.
\end{theorem}
\begin{proof}
	With the help of \eqref{generalintegral}, \eqref{L2L0}-\eqref{letole1} and Lemma \ref{theorem1}, we can get
	\begin{equation*}
		\begin{split}
			&\left|\Phi_{2,B_s}^+(\bs r_i)-\sum\limits_{m=0}^{p}\sum\limits_{n=m}^{p}\frac{n![(x_c-\tilde x_c)+\ri(y_c-\tilde y_c) ]^{n-m}}{(n-m)!m!}\beta_{n}[(\tilde x_c-x_i)+\ri(\tilde y_c-y_i) ]^{m}\right|\\
			=&\left| \sum\limits_{n=p+1}^{\infty}\sum\limits_{m=0}^{n}\frac{n![(x_c-\tilde x_c)+\ri(y_c-\tilde y_c) ]^{n-m}}{(n-m)!m!}[(\tilde x_c-x_i)+\ri(\tilde y_c-y_i) ]^{m}\beta_n\right|\\
			=&\left| \sum\limits_{n=p+1}^{\infty}[(x_c-x_i)+\ri (y_c-y_i)]^n \left(\sum\limits_{\bs r_j^{\rm im}\in B_s}Q_j\ri^{-n} \mathcal I_n(x_c-x_j, y_c-y_j^{\rm im})\right)\right|\\
			\leq&\sum\limits_{n=p+1}^{\infty}\sum\limits_{\bs r_j^{\rm im}\in B_s}|Q_j|\left|\mathcal I_n(x_c-x_j, y_c-y_j^{\rm im})\right| |\bs r_c-\bs r_i|^n\\
			\le&\sum\limits_{n=p+1}^{\infty}\Big(\frac{|\bs r_c-\bs r_i|}{\min\limits_{\bs r_j^{\rm im}\in B_s} |\bs r_c-{\bs r}_j^{\rm im}|}\Big)^n \frac{CQ_s\max\{|Z_{\varepsilon}|^2|\bs r_j^{\rm im}-\bs r_c|^2, 1\}}{n}\\
			\le&\frac{CQ_s}{p+1}\frac{\tilde{q}^{p+1}}{1-\tilde{q}}.
		\end{split}
	\end{equation*}
\end{proof}

The coefficient conversion formula \eqref{metoleimage} from multipole to local (M2L) conversion incurs errors due to truncation requirements. Therefore, we provide below an analysis of the truncation errors in the multipole to local (M2L) conversion formula.
\begin{theorem}\label{theorem6}
	The truncated ME to LE translation \eqref{truncatedM2L} has  error estimate
	\begin{equation}\label{gg1.13}
		\begin{split}
			&\left|\Phi_{2,B_s}^+(\bs r_i)-{\Phi}_{M2L}^{p}(\bs r_i)\right|
			\le\frac{CQ_s}{p+1}\left[C_{max}{q}^{p+1}+\check{C}_{max}\check{q}^{p+1}\right],
		\end{split}
	\end{equation}
	where $q$, $Q_s$ and $C_{max}$ are defined in \eqref{localcharge}, 
	$$\check{C}_{max}=\frac{2\max\{|Z_{\varepsilon}|^2|\bs r'_{c}-\bs r_{c}|^2,1\}|\bs r_c'-\bs r_c|}{2|\bs r_c'-\bs r_c|-\sqrt{2}( d_s+ d_t)},\quad \check{q}=\frac{\sqrt{2} d_t}{2|\bs r_c'-\bs r_c|-\sqrt{2} d_s}<1,$$ 
	and $C$ is a positive constant.
\end{theorem}

\begin{proof}
	By the definition \eqref{truncatedME}, \eqref{truncatedM2L} and \eqref{truncatedM2Lcoef}, we have
	\begin{equation*}
		\begin{split}
			&{\Phi}_{ME}^{p}(\bs r_i)-{\Phi}_{M2L}^{p}(\bs r_i)\\
			=&\sum\limits_{m=0}^{p}\alpha_{m}\mathcal I_m(x_i-x_c', y_i-y_c')-\sum\limits_{n=0}^{p}\sum\limits_{m=0}^{p}\frac{\ri^{-n}(n+m)!\mathcal I_{n+m}(x_c- x'_c,y_c-y'_c)\alpha_{m}}{ n!m![(x_c-x_i)+\ri (y_c-y_i)]^{-n}}.
		\end{split}    
	\end{equation*}
	The derivation of the M2L translation shows that 
	\begin{equation*}
		\mathcal I_m(x_i-x_c', y_i-y_c')=\sum\limits_{n=0}^{\infty}\frac{\ri^{-n}(n+m)!\mathcal I_{n+m}(x_c- x'_c,y_c-y'_c)}{ n!m![(x_c-x_i)+\ri (y_c-y_i)]^{-n}}.
	\end{equation*}
	Substituting this expansion back into the last equation and then taking absolute value on both sides, we obtain
	\begin{equation*}
		\begin{split}
			|{\Phi}_{ME}^{p}(\bs r_i)-{\Phi}_{M2L}^{p}(\bs r_i)|=\left|\sum\limits_{n=p+1}^{\infty}\sum\limits_{m=0}^{p}\frac{\ri^{-n}(n+m)!\mathcal I_{n+m}(x_c- x'_c,y_c-y'_c)\alpha_{m}}{ n!m![(x_c-x_i)+\ri (y_c-y_i)]^{-n}}
			\right|.
		\end{split}
	\end{equation*}
	Applying Lemma \ref{theorem1} and the estimate 
	$$|\alpha_m|\leq Q_s\Big(\frac{\sqrt{2}}{2}d_s\Big)^m,\quad |[(x_c-x_i)+\ri (y_c-y_i)]^{n}|\leq \Big(\frac{\sqrt{2}}{2}d_t\Big)^n$$
	gives estimate
	\begin{equation*}
		\begin{split}
			|{\Phi}_{ME}^{p}&(\bs r_i)-{\Phi}_{M2L}^{p}(\bs r_i)| \leq \sum\limits_{n=p+1}^{\infty}\sum\limits_{m=0}^{p} \frac{(n+m)!}{n!m!}\frac{CQ_s\max\{|Z_{\varepsilon}|^2|\bs r'_{c}-\bs r_{c}|^2,1\}\hat d_s^m\hat d_t^n}{(n+m)|\bs r_c'-\bs r_c|^{n+m}}\\
			\leq&\sum\limits_{n=p+1}^{\infty}\frac{CQ_s\max\{|Z_{\varepsilon}|^2|\bs r'_{c}-\bs r_{c}|^2,1\}}{n}\Big(\frac{\hat d_t}{|\bs r_c'-\bs r_c|}\Big)^n\sum\limits_{m=0}^{\infty} \frac{(n+m)!}{n!m!}\Big(\frac{\hat d_s}{|\bs r_c'-\bs r_c|}\Big)^m\\
			\leq&\frac{CQ_s\max\{|Z_{\varepsilon}|^2|\bs r'_{c}-\bs r_{c}|^2,1\}}{p+1}\sum\limits_{n=p+1}^{\infty}\Big(\frac{\hat d_t}{|\bs r_c'-\bs r_c|}\Big)^n\Big(\frac{|\bs r_c'-\bs r_c|}{|\bs r_c'-\bs r_c|-\hat d_s}\Big)^{n+1}\\
			=&\frac{CQ_s\max\{|Z_{\varepsilon}|^2|\bs r'_{c}-\bs r_{c}|^2,1\}}{p+1}\frac{|\bs r_c'-\bs r_c|}{|\bs r_c'-\bs r_c|-\hat d_s-\hat d_t}\Big(\frac{\hat d_t}{|\bs r_c'-\bs r_c|-\hat d_s}\Big)^{p+1},
		\end{split}
	\end{equation*}
	where $\hat d_s=\frac{\sqrt{2}}{2}d_s, \hat d_t=\frac{\sqrt{2}}{2}d_t$. Now, the triangular inequality
	\begin{equation*}
		\begin{split}
			|\Phi_{2,B_s}^+(\bs r_i)-{\Phi}_{M2L}^{p}(\bs r_i)|\leq|\Phi_{2,B_s}^+(\bs r_i)-{\Phi}_{ME}^{p}(\bs r_i)|+|{\Phi}_{ME}^{p}(\bs r_i)-{\Phi}_{M2L}^{p}(\bs r_i)|
		\end{split}
	\end{equation*}
	together with Theorem \ref{theorem2} implies the conclusion.
\end{proof}

In the FMM, the longest approximation chain is: first calculate truncated ME $\Phi_{ME}^p(\bs r_i)$ and then do a ME to ME shifting to obtain a new ME $\widetilde{\Phi}_{ME}^p(\bs r_i)$ w.r.t. to the center of the parent box. After that a M2L translation is taken to translate $\widetilde{\Phi}_{ME}^p(\bs r_i)$ to a LE ${\Phi}_{M2L}^p(\bs r_i)$ and finally do a LE shifting to obtain a new LE $\widetilde{\Phi}_{LE}^p(\bs r_i)$ w.r.t. the center of a child box. Apparently, with the error estimates presented in Theorem \ref{theorem2} to Theorem \ref{theorem6}, an overall exponential convergence w.r.t truncation number $p$ for the FMM can be obtained by simply using the triangular inequality.

\section{Numerical examples}\label{section5}
In this section, some numerical examples are given to validate the efficacy and precision of the proposed fast multipole method. 

{\bf Example 1:} This example is to verify the exponential convergence we have proved for the ME, LE expansions and their shifting (M2M, L2L) and translation (M2L) operators. Given a target-source pair located at the points $\bs r=(0.625, 1.25) $,  $\bs r'=(0, 0.375)$. Consider the far-field approximation of $ \Phi_2^+(\bs{r})$ and $ \Phi_2^-(\bs{r})$ with respect to target centers 
$$\bs r_c=(0.65625, 1.09375),\quad \tilde{\bs r}_c=(0.8125, 0.9375).$$
and source centers 
$$ \bs r_c'=(0.03125, -0.46875),\quad \tilde{\bs r}_c'=(0.1875, -0.3125).$$
Here, the image point is given by $\bs r'_{\rm im}=(0, -0.375) $. A diagram for the locations of the source, target and centers is presented in Fig. \ref{Hezi}. 

We set a charge $Q=1$ at the source point $\bm r'$. The absolute error and the theoretical estimate of the multipole expansion (ME) for $\Phi_2^+(\bs{r})$ are compared in Fig. \ref{f4.2} (a). Analogously, the error and theoretical results for the local expansion (LE) formula to calculate $\Phi_2^-(\bs{r})$ are compared in Fig. \ref{f4.2} (b). The results show that both ME and LE have exponential convergence with respect to the truncation number $p$ and our theoretical analysis gives a sharp error estimate for the approximation.  We also compare the convergence of the whole approximation used in the FMM with our theoretical results in Fig. \ref{f4.2} (c), (d).  Clearly, the proposed FMM has exponential convergence with respect to truncation number $p$ and the numerical results are consistent with the error analysis provided in  Section \ref{section4}. 
% It is important to note that the fit between the error and estimate presented in \cref{fig3} and \cref{fig4} is inferior to the fit between the error and estimate depicted in \cref{fig1} and \cref{fig2}. 
% The proximity between the error and estimate in \cref{fig3} and \cref{fig4} is not as close as that in \cref{fig1} and \cref{fig2}. 
% This discrepancy arises from the application of formula \eqref{gg1.15} to assess the truncation error of the M2L transform, leading to an amplification of the associated error estimate.
% The proximity between the error and estimate in Fig. \ref{fig3} and Fig. \ref{fig4} is not as close as that in Fig. \ref{fig1} and Fig. \ref{fig2}, and this difference is attributed to the use of an infinite series summation in place of a finite summation for the truncation error estimation of the M2L transform.
% As shown in Fig. \ref{fig1}-Fig. \ref{fig4}, the accuracy of the produced numerical results can be improved greatly by increasing the truncated terms of the FMM. Clearly, the proposed method has an exponential convergence with respect to truncated terms. Moreover, only using a few truncated terms, we can achieve the desired precision. 

\begin{figure}[h!]
	\centering
	\includegraphics[width=0.43\textwidth]{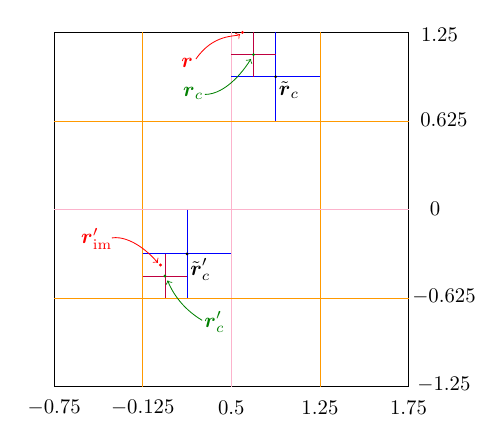}\quad
	\includegraphics[width=0.43\textwidth]{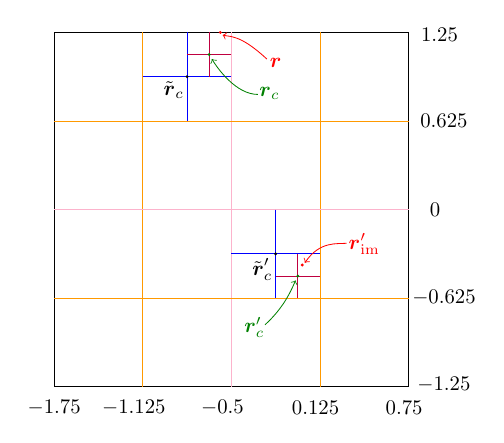}
	\caption{Diagrams of the imag points for $\Phi_2^+(\boldsymbol{r})$ (left) and $\Phi^-_2(\boldsymbol{r})$ (right).}
	\label{Hezi}
\end{figure} 

% \begin{figure}[htbp]
	% 	\centering
	% 	\begin{minipage}{0.48\linewidth}
		% 		\centering
		%         \vspace{0pt}
		% 		\includegraphics[width=0.9\linewidth]{pics/1MEfor1.eps}
		% 		\vspace{-0.8em}
		% 		\caption{ME for $\Phi_2^+(\bs{r}_{i}^{\rm im})$}
		% 		\label{fig1}%文中引用该图片代号
		% 	\end{minipage}
	% 	%\qquad
	% 	\begin{minipage}{0.48\linewidth}
		% 		\centering
		%         \vspace{0pt}
		% 		\includegraphics[width=0.9\linewidth]{pics/1LEfor2.eps}
		
		% 		\vspace{-0.8em}
		% 		\caption{ LE for $\Phi_2^-(\bs{r}_{i}^{\rm im})$}
		% 		\label{fig2}%文中引用该图片代号
		% 	\end{minipage}
	% \end{figure}
% %\vskip-2em
% \begin{figure}[htbp]
	% 	\centering
	% 	\begin{minipage}{0.48\linewidth}
		% 		\centering
		% 		\includegraphics[width=0.9\linewidth]{pics/1M2Lfor1.eps}
		
		% 		\vspace{-0.8em}
		% 		\caption{M2M$\rightarrow$M2L$\rightarrow$L2L for $\Phi_2^+(\bs{r}_{i}^{\rm im})$}
		% 		\label{fig3}%文中引用该图片代号
		% 	\end{minipage}
	% 	%\qquad
	% 	\begin{minipage}{0.48\linewidth}
		% 		\centering
		% 		\includegraphics[width=0.9\linewidth]{pics/1M2Lfor2.eps}
		% 		\vspace{-0.8em}
		% 		\caption{M2M$\rightarrow$M2L$\rightarrow$L2L for $\Phi_2^-(\bs{r}_{i}^{\rm im})$}
		% 		\label{fig4}%文中引用该图片代号
		% 	\end{minipage}
	% \end{figure} 

\begin{figure}[!ht]
	\centering
	\subfigure[ME for $\Phi_2^+(\bs{r})$]{
		\includegraphics[width=0.45\textwidth]{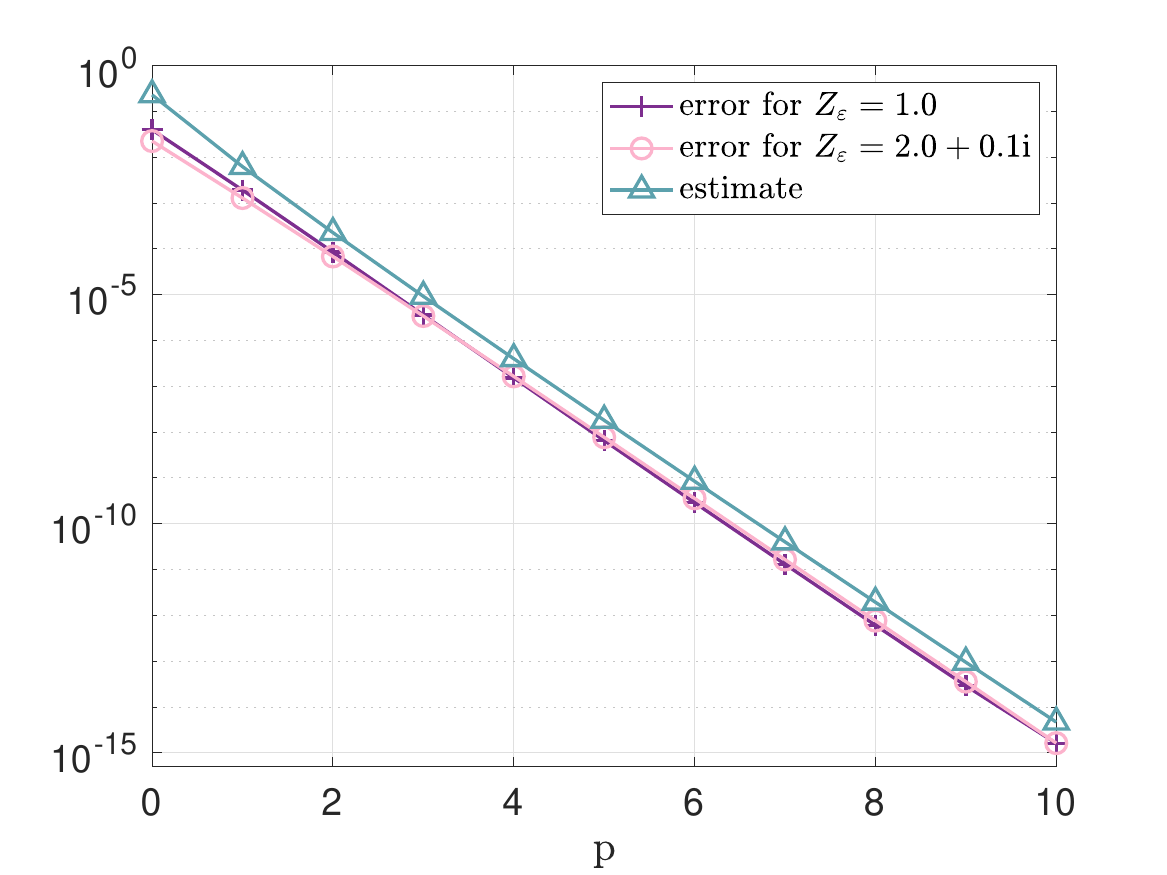}}\quad
	\subfigure[ LE for $\Phi_2^-(\bs{r})$]{
		\includegraphics[width=0.45\textwidth]{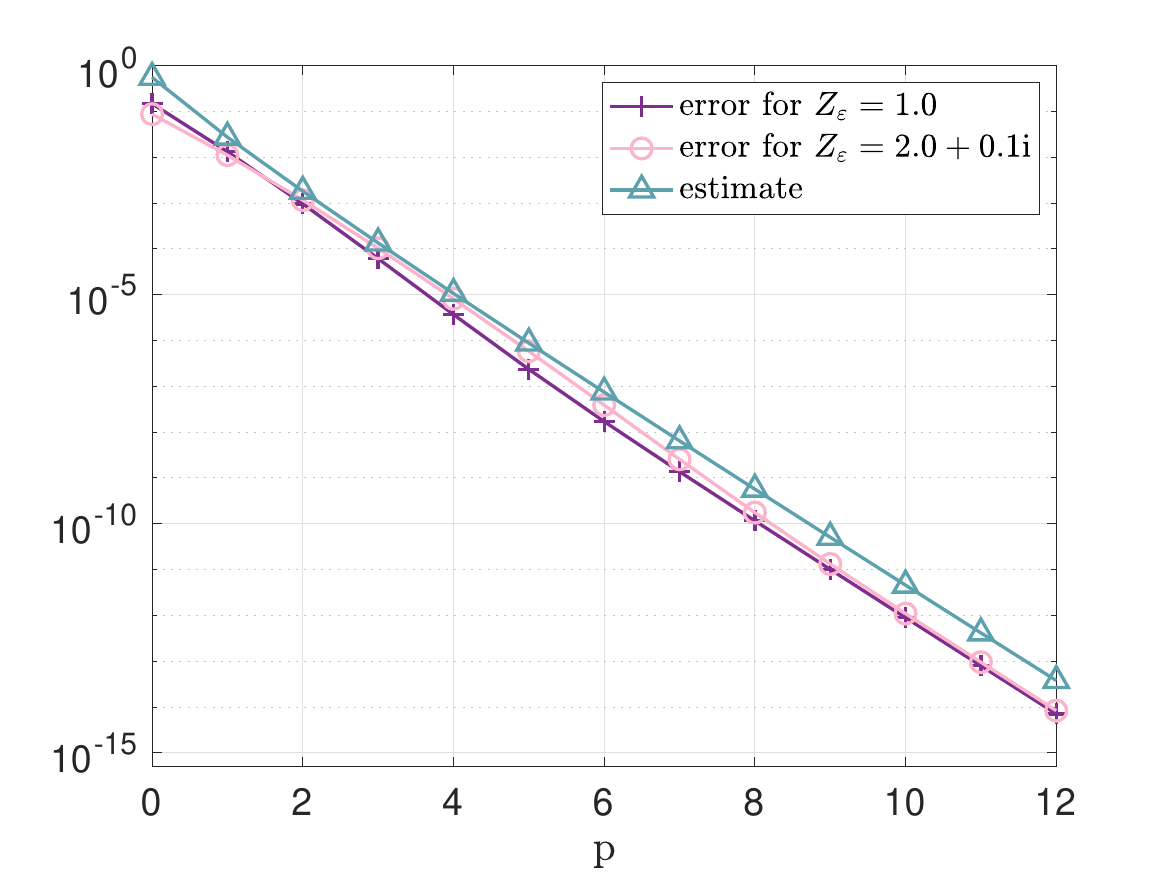}}
	\subfigure[M2M$\rightarrow$M2L$\rightarrow$L2L for $\Phi_2^+(\bs{r})$]{
		\includegraphics[width=0.45\textwidth]{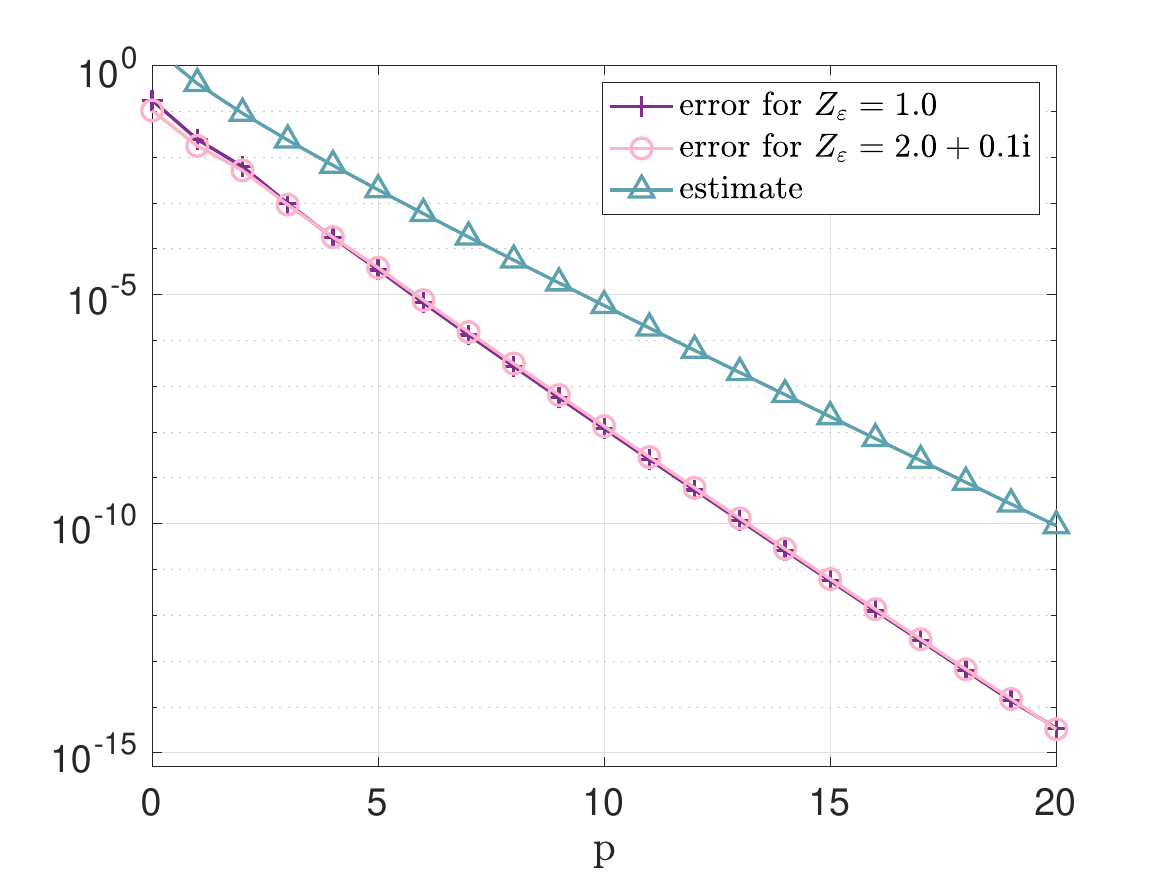}}\quad
	\subfigure[M2M$\rightarrow$M2L$\rightarrow$L2L for $\Phi_2^-(\bs{r})$]{
		\includegraphics[width=0.45\textwidth]{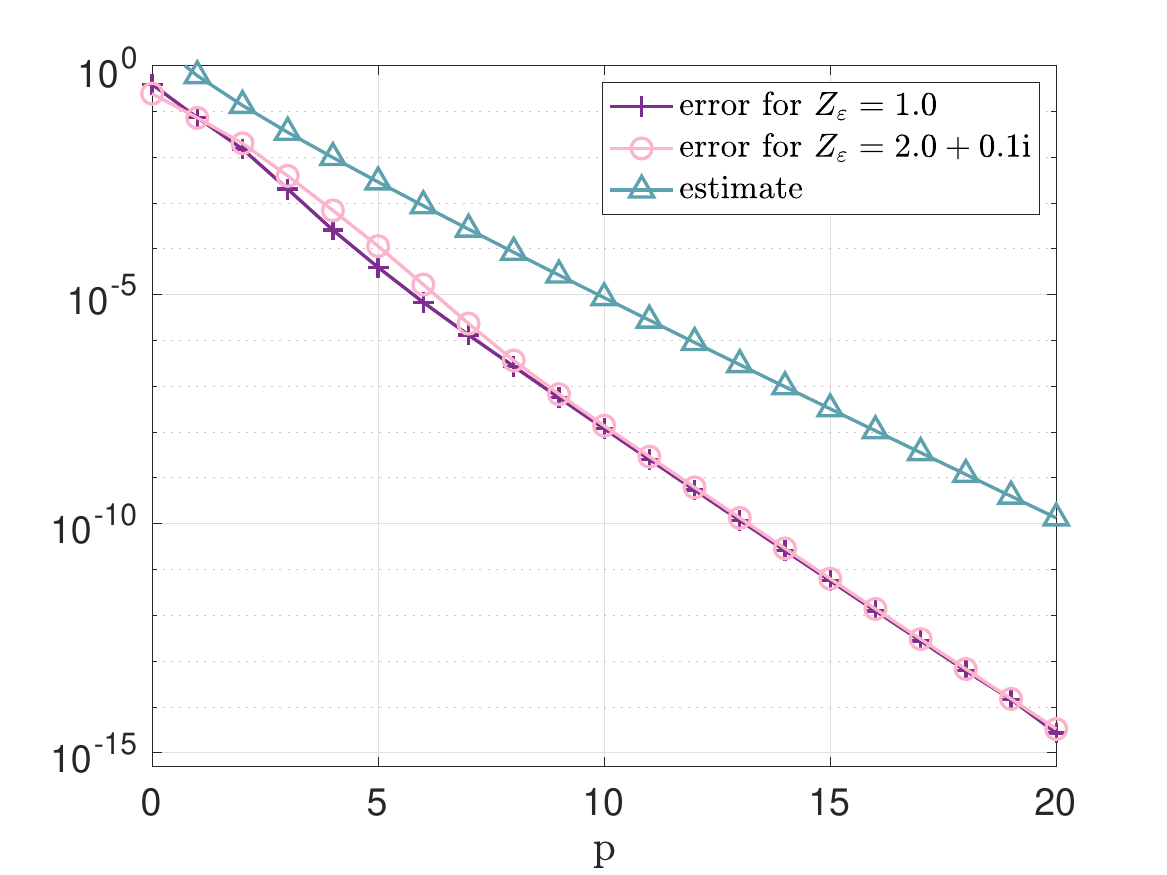}}
	\caption{Exponential convergence of the expansions and shifting and translation operators.}\label{f4.2}
\end{figure}

% \begin{figure}[h!]
	% 	\centering
	% 	\includegraphics[width=0.6\textwidth]{pics/yuan2.pdf}
	% 	\caption{Circles with charges uniformly distributed on the boundary.}
	% 	\label{yuan}
	% \end{figure} 

\begin{figure}[!ht]
	\centering
	\subfigure[Charged circles]{\includegraphics[width=0.4\textwidth]{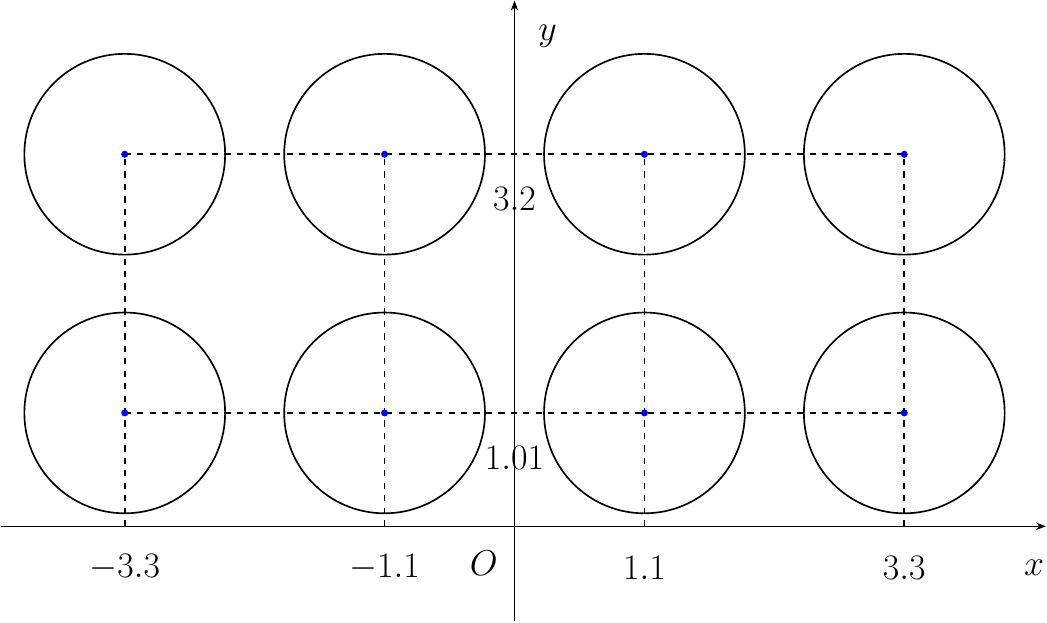}}\qquad\quad\quad
	\subfigure[$\Phi^{free}$]{\includegraphics[width=0.48\textwidth]{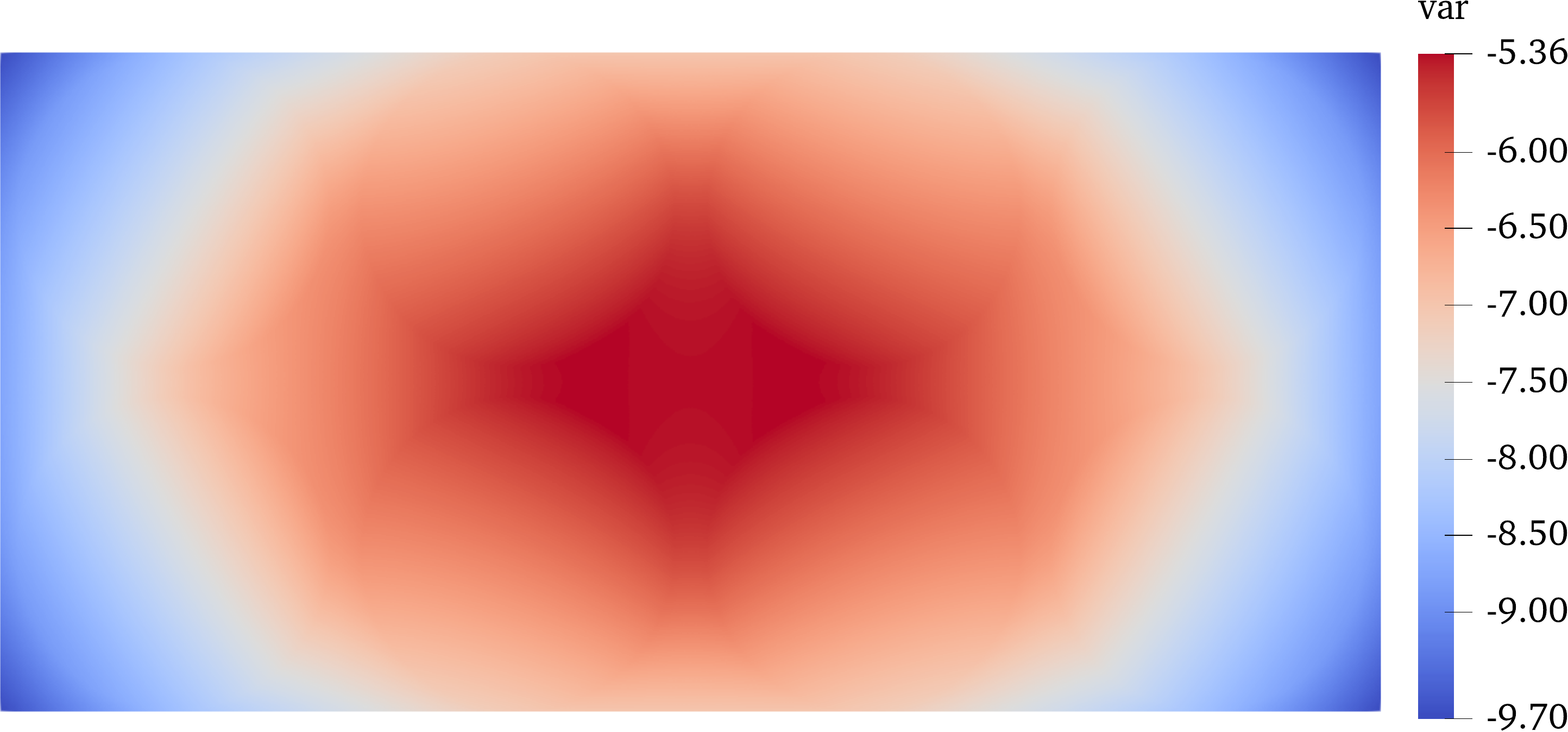}}\\
	\subfigure[Real part of $\Phi^{free}+\Phi_1+\Phi_2$]{
		\includegraphics[width=0.48\textwidth]{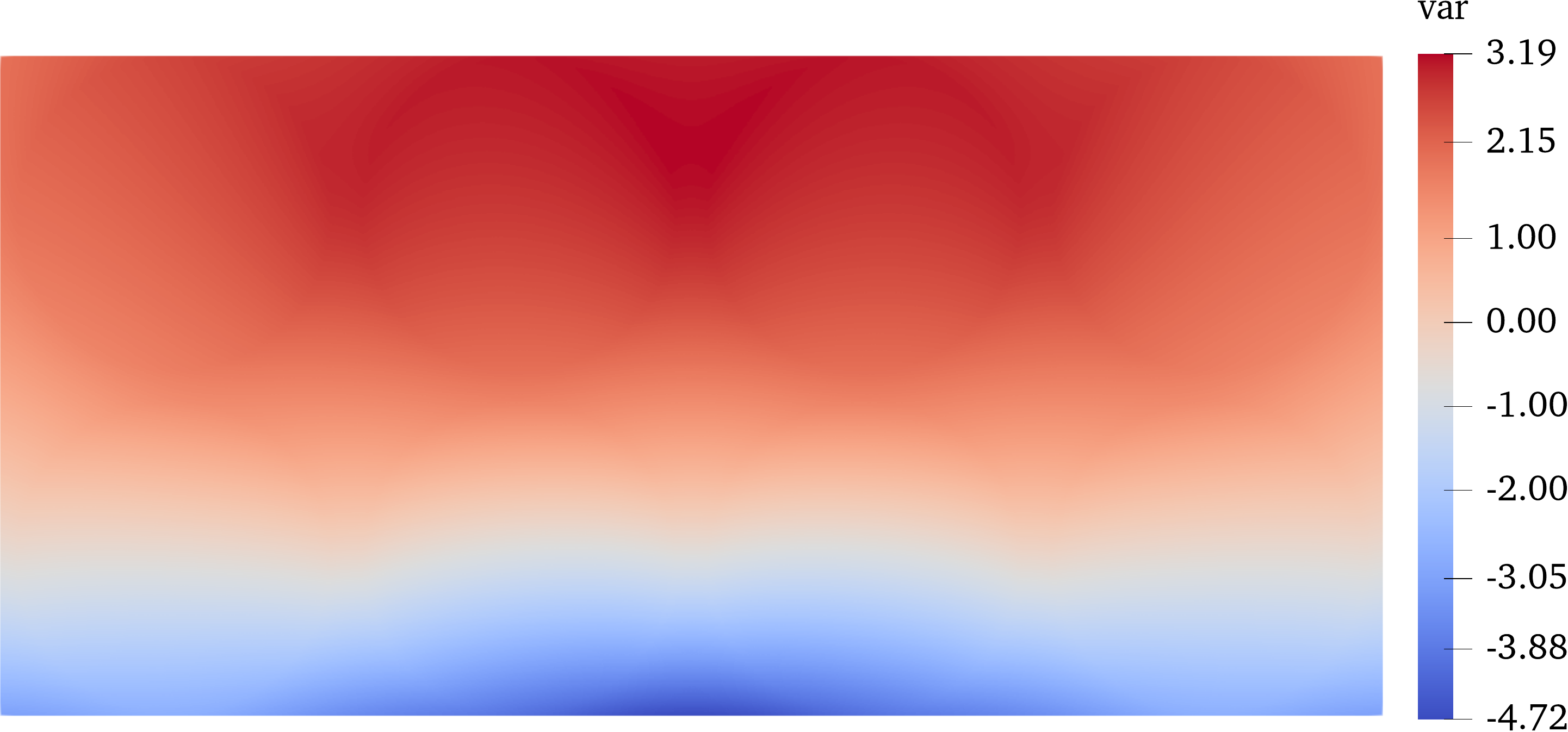}}
	\subfigure[Imaginary part of $\Phi^{free}+\Phi_1+\Phi_2$]{
		\includegraphics[width=0.48\textwidth]{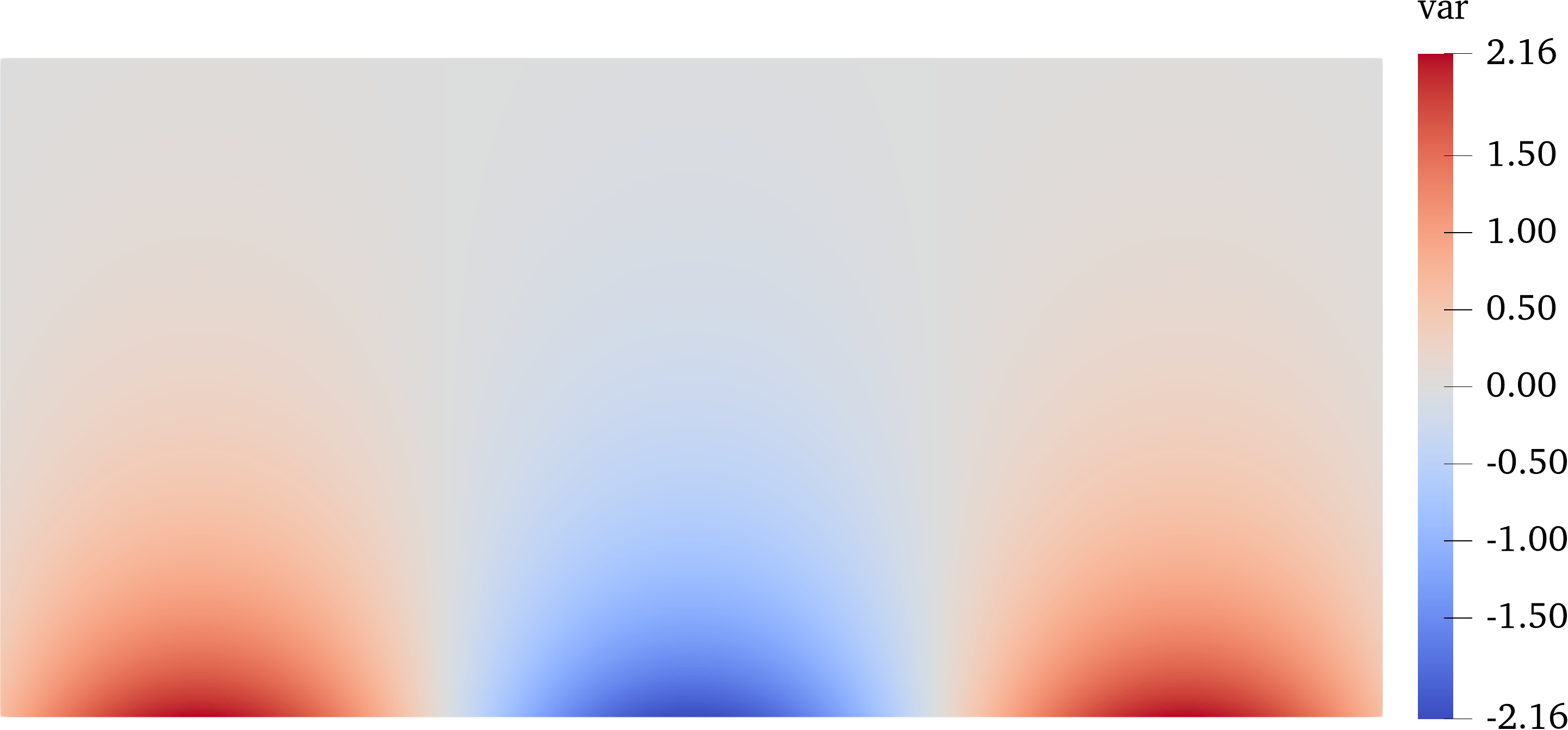}}
	\caption{A comparison between $\Phi^{free}$ and $\Phi^{free}+\Phi_1+\Phi_2$.
	}\label{val}
\end{figure}

% \begin{figure}[h!]
	% 	\centering
	% 	\includegraphics[width=0.44\textwidth]{pics/realval}
	% 	\includegraphics[width=0.44\textwidth]{pics/imagval}
	% 	\caption{ The real (left) and imaginary (right) parts of the values for the reaction field.}
	% 	\label{val}
	% \end{figure} 

% \begin{figure}[h!]
	% 	\centering
	%     \includegraphics[width=0.49\textwidth]{pics/Time4.eps}
	% 	\includegraphics[width=0.49\textwidth]{pics/Time3.eps}
	% 	\caption{CPU time (sec) vs number of partition for $N=10$.}
	% 	\label{time}
	% \end{figure} 

\begin{figure}[!ht]
	\centering
	\subfigure[Error vs p with $N^{\rm field}+8N^{\rm source}=4000$.]{
		\includegraphics[width=0.48\textwidth]{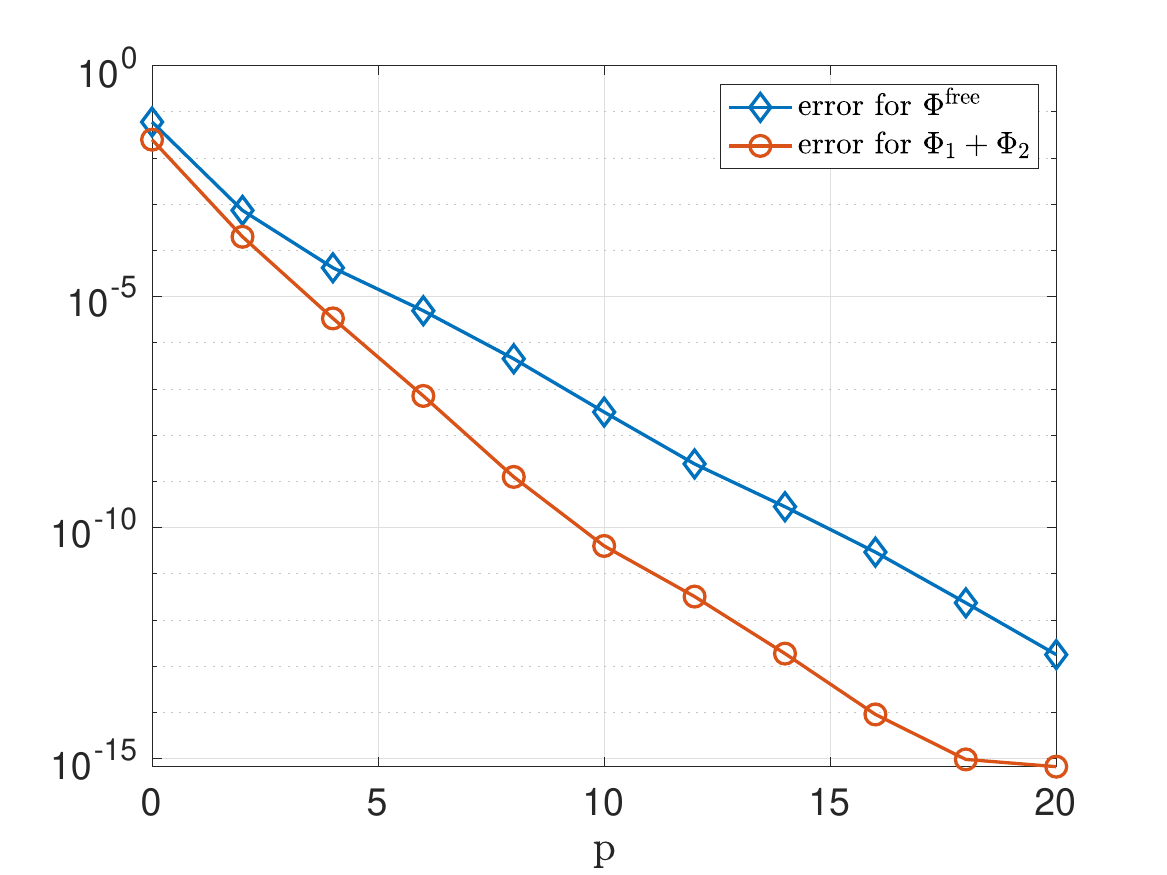}}
	\subfigure[CPU time vs $N^{\rm field}+8N^{\rm source}$ with $p=10$.]{
		\includegraphics[width=0.48\textwidth]{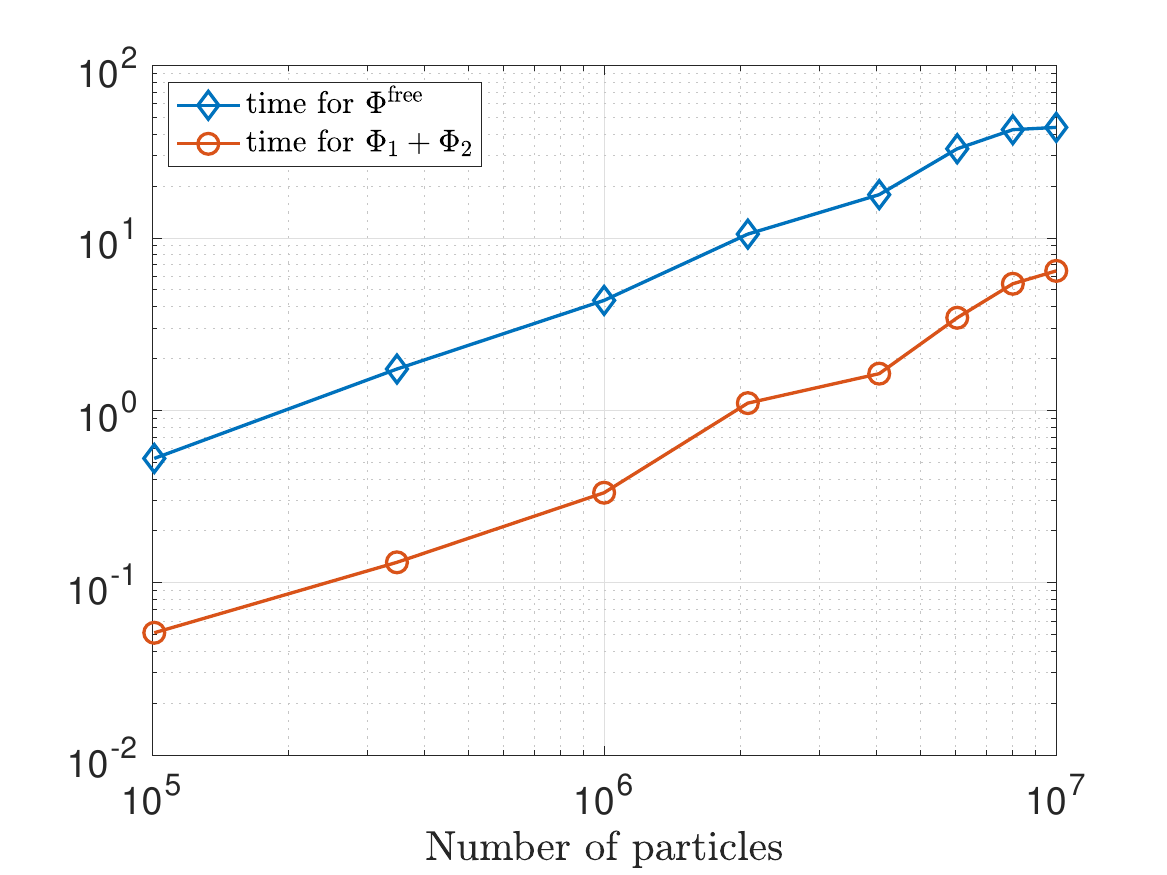}}
	\caption{Performance of the FMM.}\label{time}
\end{figure}

{\bf Example 2:} This example is used to test the efficiency and accuracy of the proposed FMM for the multiple charge interaction problem. Consider $8$ circles of radius $1$ centered on $(-3.3+2.2n, 1.01)$, $(-3.3+2.2n, 3.2)$, $n=0, 1, 2, 3$. Uniform line charges with density $\rho=0.0001$ are assumed on the circles. The lossless boundary condition with $Z_{\varepsilon}=1.0$ is imposed on $y=0$. In this example, there are charges close to the boundary $y=0$ (with distance equal to $0.01$), see Fig. \ref{val} (a) for a sketch of the configuration. Subsequently, we employ the proposed fast multipole method to compute the potential in the domain $\Omega=[-4.4, 4.4]\times[0, 4.2]$. 
Uniform Cartesian meshes are used to discretize $\Omega$ with mesh points denoted by $\{\bs r_j=(x_j, y_j)\}_{j=1}^{N_{\rm field}}$ where $N_{\rm field}$ is the number of field points. 
The circles are also discretized using uniform meshes where $\{\tilde{\bs r}_{ij}\}_{j=1}^{N_{\rm source}}$, $i=1, 2, \cdots, 8$ is the middle points of the $j$-th segment on the $i$-th circle. The free space and reaction components of the potential are then approximated by
\begin{equation*}
	\begin{split}
		&\Phi^{\rm free}(\bs r_k)=-\frac{1}{2\pi}\sum\limits_{i=1}^8\sum\limits_{j=1}^{N_{\rm source}}\rho_{ij}\ln(|\bs r_k-\tilde{\bs r}_{ij}|),\\
		&\Phi_1(\bs r_k)=\frac{1}{2\pi}\sum\limits_{i=1}^8\sum\limits_{j=1}^{N_{\rm source}}\rho_{ij}\ln(|\bs r_k-\tilde{\bs r}_{ij}^{\rm im}|),\quad \Phi_2(\bs r_k)=\sum\limits_{i=1}^8\sum\limits_{j=1}^{N_{\rm source}}\rho_{ij}G_{Z_{\varepsilon}}(\bs r_k,\tilde{\bs r}_{ij}),
	\end{split}
\end{equation*}
where $\rho_{ij}=\frac{2\pi}{N_{\rm source}}$ is the charge inside each segment. The potentials with or without the presence of an impedance boundary at $y=0$ are compared in Fig. \ref{val} and the accuracy versus the truncation number $p$ and the CPU time versus the total number of particles are plotted in Fig. \ref{time}. We can clearly see that the proposed FMM has exponential convergence and $O(N)$ complexity as the classic FMM. Moreover, the computation times for the reaction components are much shorter than that for the free space components. 

\section{Conclusion}\label{section6}
In this paper, we propose a fast multipole method for the two-dimensional Laplace equation in the half-plane with a Robin boundary condition. The algorithm is implemented by incorporating a novel far-field approximation theory for the Green's function of the half-plane problem into the framework of the classic FMM. Exponential convergence of the far-field approximation theory and a thorough error estimate of the FMM are proved. Both theoretical analysis and numerical results show that the proposed FMM for half-space problems can achieve performance similar to that of the classic FMM for free-space problems.

In future work, we will consider the simulation of water waves using our FMM together with boundary element techniques.  Moreover, we will develop the fast multipole method for the three-dimensional Laplace equation in half space with Robin boundary condition which has important applications in medical sciences and engineering.

%%%%%%%%%%%%%%

%\appendix
\section*{Appendix A. Proof of Theorem 2.1}
\begin{proof}
	We shall only give a proof for the first expansion as the other one can be proved similarly. 
	Applying the variable substitution $t=(y-\ri x)\lambda$, we have
	\begin{equation}\label{integralalongS}
		\mathcal I_0(x-x', y-y')
		=\frac{1}{2\pi}\int_{S}\frac{e^{-t}e^{{(y'-\ri x')t}/{(y-\ri x)}}}{t-Z_{\varepsilon}(y-\ri x)}dt,
	\end{equation}
	where $S$ is the contour defined as $S=\{t=(y-\ri x)\lambda: \lambda\in [0,+\infty)\}$, see Fig. \ref{f1.4}. We shall change the contour from $S$ back to the real line in the $t$-plane using Cauchy theorem. As the integrand has a simple pole $P=Z_{\varepsilon}(y-\ri x)$ in the $t$-plane, the contour deformation depends on the position of $P$ and could produce some extra term w.r.t the pole.
	\begin{figure}[!ht]
		\centering
		\subfigure[$x<0$]{
			\includegraphics[width=0.40\textwidth]{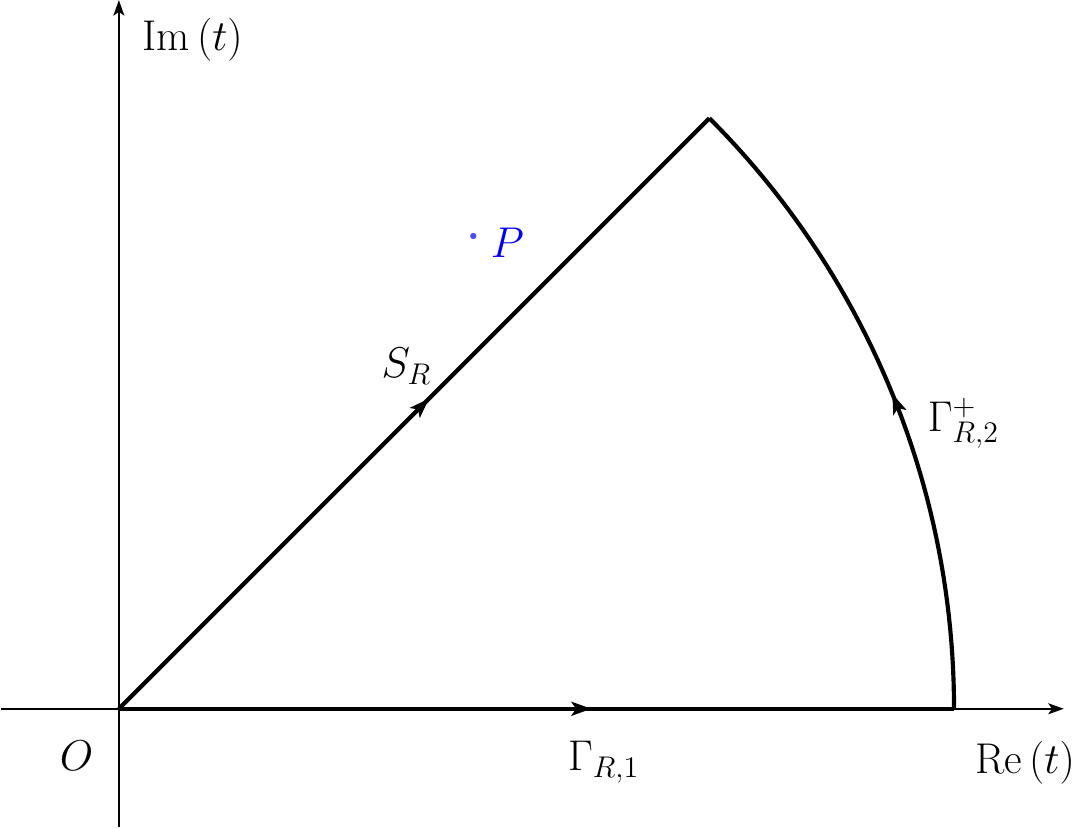}}\quad
		\subfigure[$x>0$, $\mathfrak{Im}(P)>0$]{
			\includegraphics[width=0.40\textwidth]{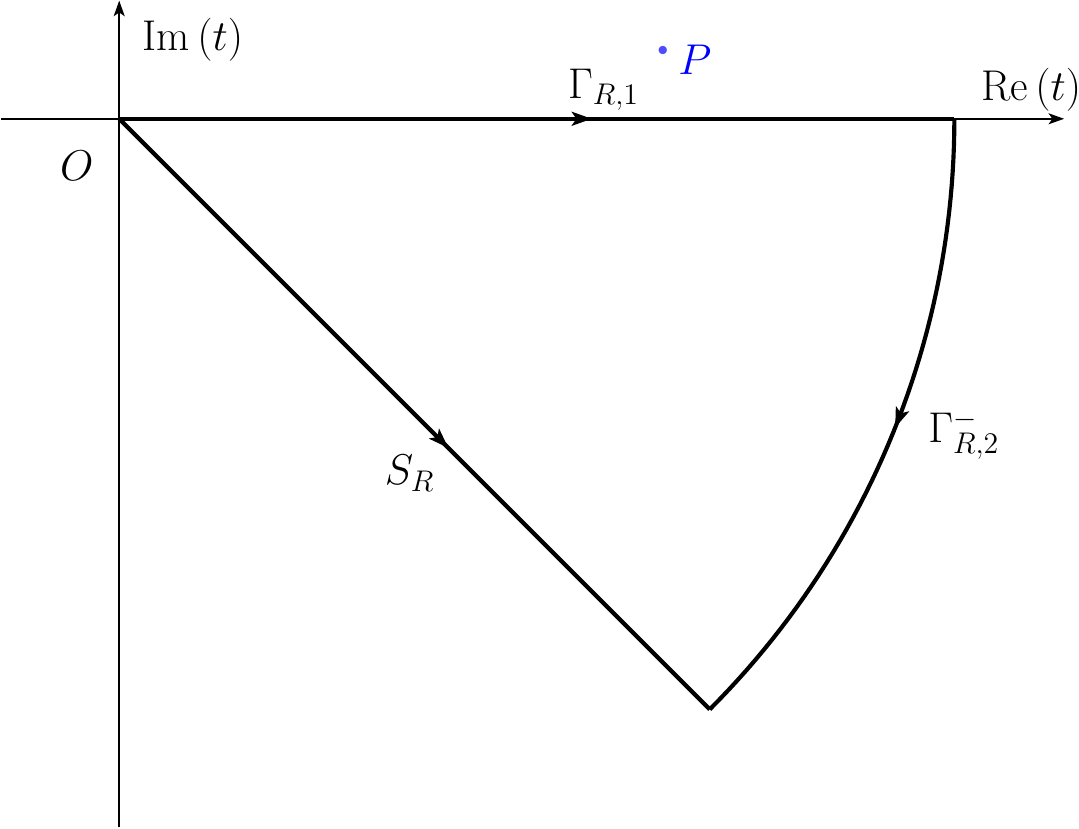}}
		\subfigure[$x>0$, $\mathfrak{Im}(P)=0$]{
			\includegraphics[width=0.40\textwidth]{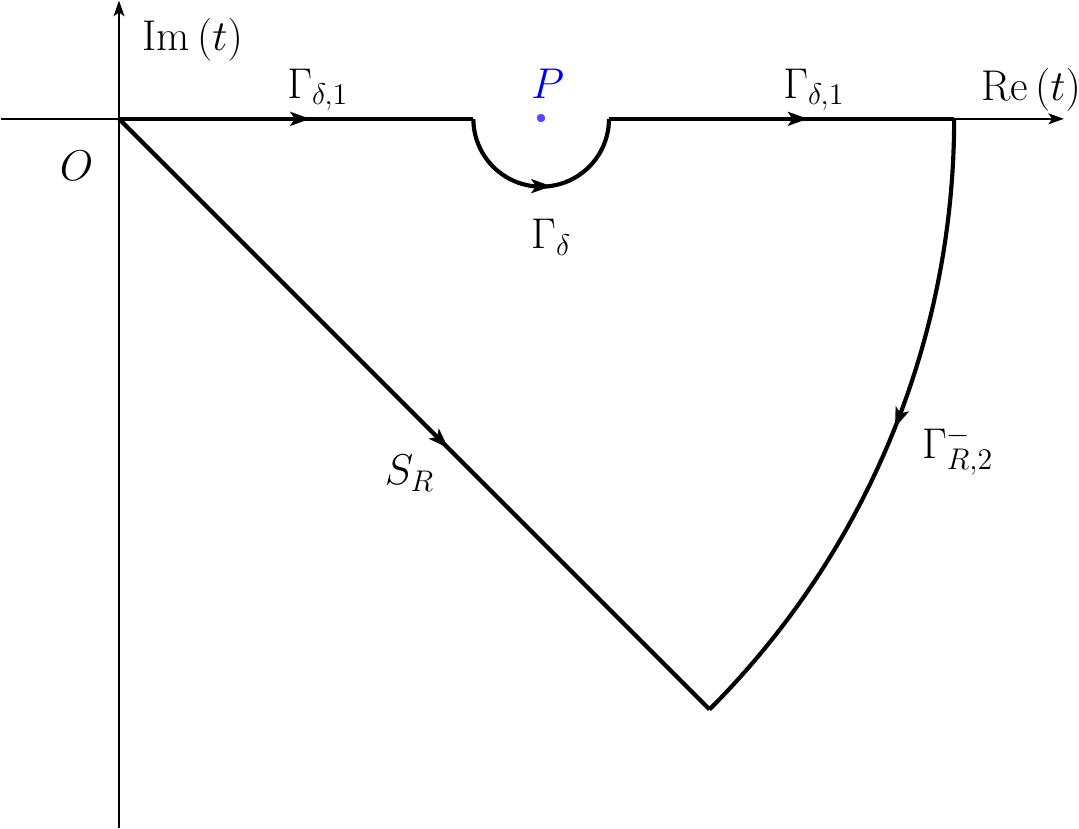}}\quad
		\subfigure[$x>0$, $\mathfrak{Im}(P)<0$]{
			\includegraphics[width=0.40\textwidth]{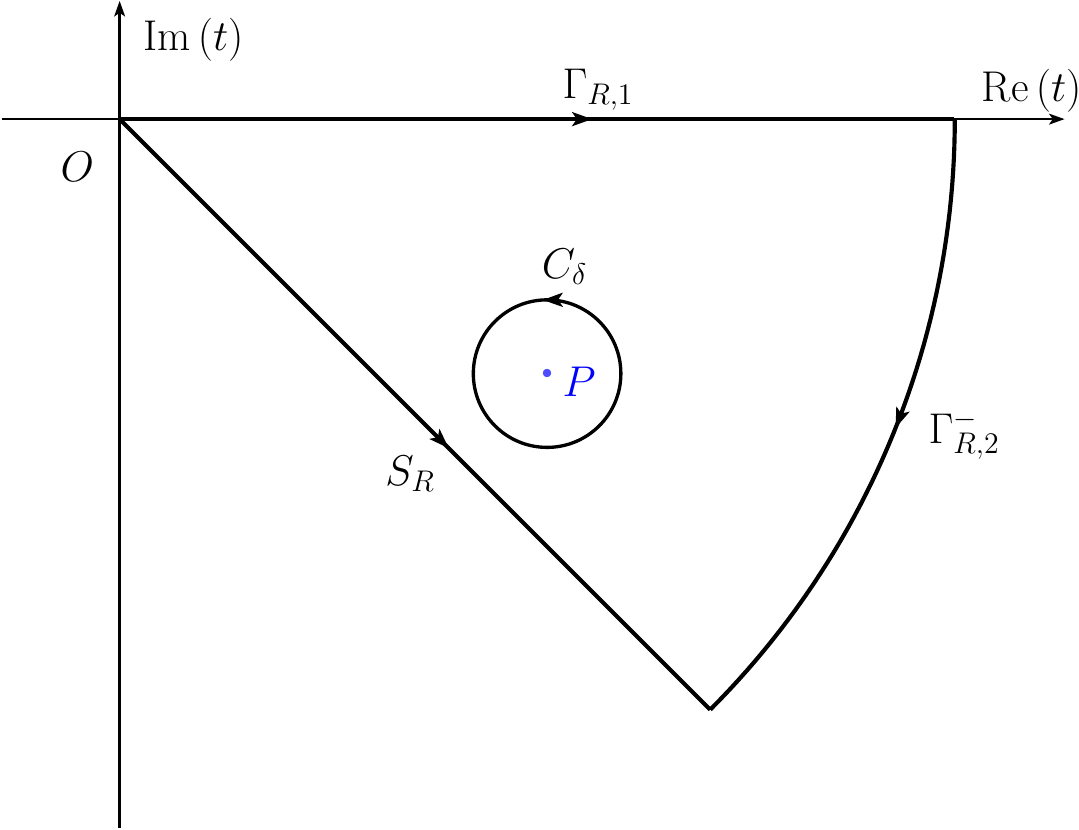}}
		\caption{Sketch of the contour change in four different cases.}\label{f1.4}
	\end{figure}
	% We only give a proof for the case $\varepsilon=0$ as the same technique can be applied to the case $\varepsilon>0$. 
	% Depends on the sign of $\mathfrak{Im}(P)$, $S$ could be in the first quadrant $(x<0)$, along the real axis $(x=0)$ and in the forth quadrant $(x>0)$, see Fig. . 
	
	Let us first consider the case $x= 0$. Integral formulation \eqref{integralalongS} can be simplified as
	\begin{equation}\label{integralrealline}
		\begin{split}
			\mathcal I_0(-x', y-y')
			=\frac{1}{2\pi}\int_{0}^{\infty}\frac{e^{-y\lambda}e^{(y'-\ri x')\lambda}}{\lambda-Z_{\varepsilon}}d\lambda=\int_{0}^{\infty}\sum\limits_{n=0}^{\infty}\frac{(y'-\ri x')^n}{2\pi y^nn!}\frac{e^{-\lambda}\lambda^n}{\lambda-yZ_{\varepsilon}}d\lambda.
		\end{split}
	\end{equation}
	Further, the assumption $|\bs r|>|\bs r'|$ gives
	\begin{equation*}
		\begin{split}
			\sum\limits_{n=0}^{\infty}\int_{0}^{\infty}\Big|\frac{(y'-\ri x')^n}{y^nn!}\frac{e^{-\lambda}\lambda^n}{\lambda-yZ_{\varepsilon}}\Big|d\lambda\leq\frac{1}{\varepsilon y}\sum\limits_{n=0}^{\infty}\Big|\frac{(y'-\ri x')^n}{y^n}\Big|=\frac{1}{\varepsilon y}\sum\limits_{n=0}^{\infty}\Big(\frac{|\bs r'|}{|\bs r|}\Big)^n<\infty.
		\end{split}
	\end{equation*}
	Then, the Fubini's theorem shows that the improper integral and infinite sum in \eqref{integralrealline} can exchange order which gives
	\begin{equation*}
		\mathcal I_0(-x', y-y')=\sum\limits_{n=0}^{\infty}\frac{(y'-\ri x')^n}{2\pi n!}\int_{0}^{\infty}\frac{e^{-\lambda}\lambda^n}{y^n(\lambda-yZ_{\varepsilon})}d\lambda=\sum\limits_{n=0}^{\infty}\ri^{-n}(x'+\ri y')^n\mathcal I_n(0, y).
	\end{equation*}

	For the case $x< 0$, the assumption $y, Z, \varepsilon >0$, we have inequality
	\begin{equation}
		\frac{\mathfrak{Im}(P)}{\mathfrak{Re}(P)}=\frac{-x+\frac{\varepsilon y}{Z}}{y+\frac{\varepsilon x}{Z}}>\frac{-x}{y},\quad {\rm if}\;\;y+\frac{\varepsilon x}{Z}>0.
	\end{equation}
	Then, the scenario is that either $\mathfrak{Re}(P)<0$ or $\mathfrak{Im}(P)>-x\mathfrak{Re}(P)/y$, i.e., $S$ is in the first quadrant and the point given by $P$ is located above the contour $S$ or in the left complex plane, see Fig. \ref{f1.4} (a).
	Therefore, the Cauchy's theorem can be applied to change the contour from $S_R:=\{t=(y-\ri x)\lambda: \lambda\in[0, R]\}$ to $\Gamma_{R,1}\cup \Gamma_{R, 2}^+$, where
	\begin{align}
		\Gamma_{R,1}=\left\{t:0\le t \le R\right\},\quad
		\Gamma_{R,2}^+=\left\{t=Re^{\ri \theta} :0\le \theta \le \theta_{xy}:=\arctan \Big(-\frac{x}{y}\Big)\right\}.  \label{contourpositive}
	\end{align}

	To discuss the integral along $\Gamma_{R, 2}^+$, we let $t=Re^{\ri\theta}$ in the exponential functions of the integrand of \eqref{integralalongS} and then take modulus that
	\begin{equation*}
		\begin{split}
			|e^{-t}e^{{(y'-\ri x')t}/{(y-\ri x)}}|=&\big|e^{-Re^{\ri \theta}\left(1-(y'-\ri x')/(y-\ri x)\right)}\big|
			=\Big|e^{-Re^{\ri \theta}\left[1-\frac{yy'+xx'+\ri(xy'-x'y)}{|\bs r|^2}\right]}\Big|\\
			=&e^{-R\left[(1-\frac{yy'+xx'}{|\bs r|^2})\cos\theta+\frac{xy'-x'y}{|\bs r|^2}\sin\theta\right]}.
		\end{split}
	\end{equation*}
	Note that $|\bs r|>|\bs r'|$, $0\leq\theta_{xy}<\frac{\pi}{2}$ for $x< 0, y>0$. Therefore, we have estimate 
	\begin{equation*}
		(1-\frac{yy'+xx'}{|\bs r|^2})\cos\theta+\frac{xy'-x'y}{|\bs r|^2}\sin\theta\geq (1-\frac{yy'+xx'}{|\bs r|^2})\cos\theta_{xy},\quad\forall \theta\in [0, \theta_{xy}]
	\end{equation*}
	for $xy'-x'y\geq 0$, and 
	\begin{equation*}
		\begin{split}
			(1-\frac{yy'+xx'}{|\bs r|^2})\cos\theta+\frac{xy'-x'y}{|\bs r|^2}\sin\theta\geq& (1-\frac{yy'+xx'}{|\bs r|^2})\cos\theta_{xy}+\frac{xy'-x'y}{|\bs r|^2}\sin\theta_{xy}\\
			=&\frac{y}{|\bs r|}-\frac{y^2y'+xx'y}{|\bs r|^3}+\frac{xx'y-x^2y'}{|\bs r|^3}=\frac{y-y'}{|\bs r|},
		\end{split}
	\end{equation*}
	for $xy'-x'y<0$ and $\theta\in [0,\theta_{xy}]$. Consequently, we obtain
	\begin{equation*}
		|e^{-t}e^{{(y'-\ri x')t}/{(y-\ri x)}}|\leq   \max\Big\{e^{-\frac{Ry}{|\bs r|}(1-\frac{yy'+xx'}{|\bs r|^2})},e^{-\frac{R(y-y')}{|\bs r|}}\Big\} ,\quad \forall t\in\Gamma_{R,2}^+,
	\end{equation*}
	which further implies that
	\begin{equation}\label{integralvanish}
		\begin{split}
			&\lim\limits_{R\rightarrow+\infty}\left|\frac{1}{2\pi}\int_{\Gamma_{R,2}^+}\frac{e^{-t}e^{{(y'-\ri x')t}/{(y-\ri x)}}}{t-P}dt\right|\\
			\le&\frac{1}{2\pi}\lim\limits_{R\rightarrow+\infty}\max\{e^{-\frac{Ry}{|\bs r|}(1-\frac{yy'+xx'}{|\bs r|^2})},e^{-\frac{R(y-y')}{|\bs r|}}\}  \int_{0}^{\theta_{xy}}\frac{1}{\left|e^{\ri \theta}-P/R\right|}d\theta=0.
		\end{split}
	\end{equation}
	
	By Cauchy's theorem and the power series of the exponential function, we obtain
	\begin{equation}\label{ee1.1}
		\mathcal I_0(x-x', y-y')
		=\frac{1}{2\pi}\int_{0}^{+\infty}\frac{e^{\frac{(y'-\ri x')t}{(y-\ri x)}-t}}{t-P}dt=\frac{1}{2\pi}\int_{0}^{+\infty}\sum\limits_{n=0}^{\infty}\frac{(y'-\ri x')^n}{n!(y-\ri x)^{n}}\frac{e^{-t}t^{n}}{t-P}dt,
	\end{equation}
	where $\mathfrak{Im}(P)=\varepsilon y-Zx>0$. Note that the assumption $|\bs r|>|\bs r'|$ implies
	\begin{equation*}
		\begin{split}
			&\sum\limits_{n=0}^{\infty}\int_{0}^{+\infty}\Big|\frac{(y'-\ri x')^n}{n!(y-\ri x)^{n}}\frac{e^{-t}t^{n}}{t-P}\Big|dt
			\leq \sum\limits_{n=0}^{\infty}\frac{|\bs r'|^n}{n!|\bs r|^{n}}\int_{0}^{+\infty}\frac{e^{-t}t^{n}}{|t-P|}dt\leq \frac{1}{\mathfrak{Im}(P)}\sum\limits_{n=0}^{\infty}\frac{|\bs r'|^n}{|\bs r|^{n}}<\infty.
		\end{split}
	\end{equation*}
	% Therefore,
	% \begin{equation}\label{estimateintegralalongreal}
		%     \sum\limits_{n=0}^{\infty}\frac{1}{2\pi}\int_{0}^{+\infty}\Big|\frac{(y'-\ri x')^n}{n!(y-\ri x)^{n}}\frac{e^{-t}t^{n}}{t-P}\Big|dt\le\frac{|\bs r|}{2\pi \mathfrak{Im}(P)(|\bs r|-|\bs r'|)}.
		% \end{equation}
	Use Fubini's theorem to exchange the order of the summation and improper integral in \eqref{ee1.1}, we get 
	\begin{equation}\label{ooo1.0}
		\mathcal I_0(x-x', y-y')
		= \sum\limits_{n=0}^{\infty}\frac{1}{2\pi}\frac{(y'-\ri x')^n}{n!(y-\ri x)^{n}}\int_{0}^{+\infty}\frac{e^{-t}t^n}{t-P}dt.
	\end{equation}
	As
	\begin{align}
		\lim\limits_{R\rightarrow+\infty}\Big|\int_{\Gamma_{R,2}^+}\frac{e^{-t}t^{n}}{t-P}dt\Big|
		=&\lim\limits_{R\rightarrow+\infty}\Big|\int_{0}^{\theta_{xy}}\frac{e^{-Re^{\ri \theta}}(Re^{\ri \theta})^nRe^{\ri \theta}\ri}{Re^{\ri \theta}-P}d\theta\Big|\\
		\leq&\lim\limits_{R\rightarrow+\infty}e^{-\frac{Ry}{|\bs r|}}R^{n}\int_{0}^{\theta_{xy}}\frac{1}{|e^{\ri\theta}-P/R|} d\theta=0,\nonumber
	\end{align}
	we can change the contour in each term of \eqref{ooo1.0} back to $S$ to obtain
	\begin{equation}\label{changecontourback}
		\begin{split}
			\mathcal I_0(x-x', y-y')
			=& \sum\limits_{n=0}^{\infty}\frac{1}{2\pi}\int_{S}\frac{(y'-\ri x')^n}{n!(y-\ri x)^{n}}\frac{e^{-t}t^n}{t-P}dt\\
			=& \sum\limits_{n=0}^{\infty}\frac{(y'-\ri x')^n}{2\pi n!}\int_{0}^{+\infty}\frac{e^{-(y-\ri x)\lambda}\lambda^n}{\lambda-Z_{\varepsilon}}d\lambda.
		\end{split}
	\end{equation}
	
	Next, we discuss the case $x>0$. In this case, we always have $\mathfrak{Re}(P)=Zy+\varepsilon x>0$. However, $\mathfrak{Im}(P)=-Zx+y\varepsilon$ can be any number in $\mathbb R$. For the case $\mathfrak{Im}(P)>0$, as shown in Fig. \ref{f1.4} (b), the proof is analogous to that of the case $x<0$ and will be omitted here for brevity.
	% \begin{figure}[!ht]
		% \centering
		% \subfigure[$x<0$]{
			% \includegraphics[width=0.40\textwidth]{pics/T1.pdf}}
		% \subfigure[$x>0$]{
			% \includegraphics[width=0.40\textwidth]{pics/T2.pdf}}
		% \caption{Integration paths for $\varepsilon=0$}\label{o1.4}
		% \end{figure}
	
	If $x>0,\mathfrak{Im}(P)=0$, the contour is sketched in Fig. \ref{f1.4} (c). Following the proof for \eqref{integralvanish}, we can verify that the integral along the contour 
	\begin{equation}\label{contournegative}
		\Gamma_{R, 2}^-=\left\{t=Re^{\ri \theta} :\theta_{xy}:=\arctan \Big(-\frac{x}{y}\Big)\le \theta \le 0\right\}
	\end{equation}
	tends to $0$ as $R\rightarrow+\infty$.
	Therefore, the Cauchy theorem and the power series of the exponential function give
	\begin{equation}\label{o1.11}
		\begin{split}
			&\frac{1}{2\pi}\int_{S}\frac{e^{-t}e^{{(y'-\ri x')t}/{(y-\ri x)}}}{t-P}dt\\
			=&\frac{1}{2\pi}\int_{\Gamma_{\delta, 1}}\sum\limits_{n=0}^{\infty}\frac{(y'-\ri x')^n}{n!(y-\ri x)^{n}}\frac{e^{-t}t^{n}}{t-P}dt+\frac{1}{2\pi}\int_{\Gamma_{\delta}}\sum\limits_{n=0}^{\infty}\frac{(y'-\ri x')^n}{n!(y-\ri x)^{n}}\frac{e^{-t}t^{n}}{t-P}dt,
		\end{split}
	\end{equation}
	where
	\begin{equation}\label{o1.12}
		\Gamma_{\delta,1}=[0, P-\delta]\cup [P+\delta,+\infty),\quad
		\Gamma_{\delta }=\left\{t=\delta e^{\ri \theta}+\mathfrak{Re}(P) :-\pi\le \theta \le 0 \right\},
	\end{equation}
	and $\delta>0$ is a small positive number.
	For fixed $\delta>0$, direct calculation leads to
	\begin{equation}\label{ee1.3}
		\sum\limits_{n=0}^{\infty}\frac{1}{2\pi}\int_{\Gamma_{\delta,1}}\Big|\frac{(y'-\ri x')^n}{n!(y-\ri x)^{n}}\frac{e^{-t}t^{n}}{t-P}\Big|dt\le\frac{|\bs r|}{2\pi \delta (|\bs r|-|\bs r'|)}.
	\end{equation}
	The integral along contour $\Gamma_{\delta}$ can reformulated as
	\begin{equation*}
		\begin{split}
			&\frac{1}{2\pi}\int_{\Gamma_{\delta}}\sum\limits_{n=0}^{\infty}\frac{(y'-\ri x')^n}{n!(y-\ri x)^{n}}\frac{e^{-t}t^{n}}{t-P}dt
			=\frac{1}{2\pi}\int_{-\pi}^{0}\sum\limits_{n=0}^{\infty}\frac{(y'-\ri x')^n}{n!(y-\ri x)^{n}}e^{-(\delta e^{\ri \theta}+P)}(\delta e^{\ri \theta}+P)^n\ri d\theta.
		\end{split}
	\end{equation*}
	Together with the estimate
	\begin{equation}
		\int_{-\pi}^{0}|e^{-(\delta e^{\ri \theta}+P)}(\delta e^{\ri \theta}+P)^n\ri |d\theta\leq \int_{-\pi}^{0}|e^{-(\delta e^{\ri \theta}+P)}|n!e^{|\delta e^{\ri \theta}+P |}d\theta\leq n!\pi e^{2\delta}
	\end{equation}
	gives
	\begin{equation}\label{absoluteconvergence1}
		\sum\limits_{n=0}^{\infty}\frac{1}{2\pi}\int_{-\pi}^{0}\Big|\frac{(y'-\ri x')^n}{n!(y-\ri x)^{n}}\frac{e^{-t}t^{n}}{t-P}\Big|d\theta\le\frac{e^{2\delta}|\bs r|}{2(|\bs r|-|\bs r'|)}.
	\end{equation}
	The convergence in \eqref{ee1.3} and \eqref{absoluteconvergence1} shows that we can apply Fubini's theorem in \eqref{o1.11} to exchange the order of the summation and integral. Then, following the proof in \eqref{ooo1.0}-\eqref{changecontourback} to change the contour of each integral term back to $S$ gives the conclusion \eqref{reactfieldME}. 
	
	If $x>0,\mathfrak{Im}(P)<0$, $P$ is located within the region enclosed by $S_
	R\cup\Gamma_{R, 1}\cup\Gamma_{R,2}^-$, as sketched in  Fig. \ref{f1.4} (d). Here, $\Gamma_{R,1}$ and $\Gamma_{R,2}^-$ are defined in \eqref{contourpositive} and \eqref{contournegative}, respectively. Also, we can verify that integral along $\Gamma_{R,2}^-$ tends to $0$ as $R\rightarrow+\infty$ and the contour \eqref{integralalongS} can be changed from $S_R$
	to $\Gamma_{R,1}\cup C_{\delta}$, i.e.,
	\begin{equation}\label{casethree}
		\begin{split}
			&\frac{1}{2\pi}\int_{S_R}\frac{e^{-t}e^{{(y'-\ri x')t}/{(y-\ri x)}}}{t-P}dt\\
			=&\frac{1}{2\pi}\int_{\Gamma_{R, 1}}\sum\limits_{n=0}^{\infty}\frac{(y'-\ri x')^n}{n!(y-\ri x)^{n}}\frac{e^{-t}t^{n}}{t-P}dt+\frac{1}{2\pi}\int_{C_{\delta}}\sum\limits_{n=0}^{\infty}\frac{(y'-\ri x')^n}{n!(y-\ri x)^{n}}\frac{e^{-t}t^{n}}{t-P}dt,
		\end{split}
	\end{equation}
	where
	\begin{equation}\label{o1.16}
		C_{\delta}=\left\{t=\delta e^{\ri \theta}+P :0\le \theta \le 2\pi \right\},
	\end{equation}
	and $\delta>0$ is a small number. 
	The proof in \eqref{ooo1.0}-\eqref{ee1.1} shows that
	\begin{equation}\label{ooo1.3}
		\lim\limits_{R\rightarrow+\infty}\frac{1}{2\pi}\int_{\Gamma_{R, 1}}\sum\limits_{n=0}^{\infty}\frac{(y'-\ri x')^n}{n!(y-\ri x)^{n}}\frac{e^{-t}t^n}{t-P}dt= \sum\limits_{n=0}^{\infty}\frac{1}{2\pi}\int_{0}^{+\infty}\frac{(y'-\ri x')^n}{n!(y-\ri x)^{n}}\frac{e^{-t}t^n}{t-P}dt.
	\end{equation}
	Moreover, mimic the proof of \eqref{absoluteconvergence1} gives
	\begin{equation}\label{ooo1.4}
		\frac{1}{2\pi}\int_{C_{\delta}}\sum\limits_{n=0}^{\infty}\frac{(y'-\ri x')^n}{n!(y-\ri x)^{n}}\frac{e^{-t}t^n}{t-P}dt= \sum\limits_{n=0}^{\infty}\frac{1}{2\pi}\int_{C_{\delta}}\frac{(y'-\ri x')^n}{n!(y-\ri x)^{n}}\frac{e^{-t}t^n}{t-P}dt.
	\end{equation}
	Let $R\rightarrow+\infty$ and using \eqref{ooo1.3} and \eqref{ooo1.4} in \eqref{casethree}, we obtain
	\begin{equation}
		\frac{1}{2\pi}\int_{S}\frac{e^{-t}e^{{(y'-\ri x')t}/{(y-\ri x)}}}{t-P}dt
		=\sum\limits_{n=0}^{\infty}\left[ \frac{1}{2\pi}\lim\limits_{R\rightarrow+\infty}\int_{\Gamma_{R,1}\cup C_{\delta}}\frac{(y'-\ri x')^n}{n!(y-\ri x)^{n}}\frac{e^{-t}t^n}{t-P}dt\right]
	\end{equation}
	Changing the contour of each integral term back to $S$ shows that conclusion \eqref{reactfieldME} also holds in this case.
\end{proof}

\section*{Appendix B. Proof of Lemma 3.1}
\begin{proof}
	If $a\leq 0$, the definition of Gamma function $\Gamma(x)$ directly gives
	\begin{equation}\label{f1.73}
		\begin{split}
			\left| \int_{0}^{+\infty}\frac{e^{-t}t^{n}}{t-z}dt \right|
			\le&\int_{0}^{+\infty}\frac{e^{-t}t^{n}}{\sqrt{(t-a)^2+b^2}}dt
			\le\int_{0}^{+\infty}e^{-t}t^{n-1}dt=(n-1)!.
		\end{split}
	\end{equation}
	For any $a>0$, we first consider the case $b> 0$. By the equality
	\begin{equation}
		\frac{t^n}{t-z}=t^{n-1}+z\frac{t^{n-1}}{t-z}=\cdots=\sum\limits_{k=1}^{n}z^{n-k}t^{k-1}+\frac{z^n}{t-z},
	\end{equation} 
	we can arrive at an estimate as follows
	\begin{equation}
		\begin{split}
			\label{integralestcase1}
			\int_{0}^{+\infty}\frac{e^{-t}t^{n}}{t-z}dt 
			=&\int_{0}^{+\infty}e^{-t}\left(\sum\limits_{k=1}^{n}z^{n-k}t^{k-1}+\frac{z^n}{t-z} \right)dt\\
			=&\sum\limits_{k=1}^{n}z^{n-k}(k-1)!+z^{n}\int_{0}^{+\infty}\frac{e^{-t}}{t-z}dt  \\
			\le&\sum\limits_{k=1}^{n}|z|^{n-k}(k-1)!
			+|z|^{n}\left|\int_{0}^{+\infty}\frac{e^{-t}}{t-z}dt\right|.
		\end{split}
	\end{equation}
	To give an estimate to the simplified integral, we split it into two integrals
	\begin{equation}\label{integraldecompose}
		\int_{0}^{+\infty}\frac{e^{-t}}{t-z}dt  =\int_{0}^{a+1}\frac{e^{-t}}{t-z}dt+\int_{a+1}^{+\infty}\frac{e^{-t}}{t-z}dt. 
	\end{equation}
	The second integral has estimate
	\begin{equation}\label{f1.717}
		\begin{split}
			\left|\int_{a+1}^{+\infty}\frac{e^{-t}}{t-z}dt\right|\le&\int_{a+1}^{+\infty}\frac{e^{-t}}{|t-z|}dt
			\le \int_{a+1}^{+\infty}e^{-t}dt \le1.
		\end{split}
	\end{equation}
	For the first integral, we change the contour to $\Gamma_1\cup\Gamma_2\cup \Gamma_3$ where
	\begin{equation}\label{f1.712}
		\begin{split}
			&\Gamma_{1}=\left\{t=\ri \eta: -1\le \eta \le 0 \right\},\quad \Gamma_{2}=\left\{t=\eta-\ri: 0\le \eta \le  a+1\right\},\\
			&\Gamma_{3}=\left\{t= a+1+\ri\eta: -1\le \eta \le 0\right\}.
		\end{split}
	\end{equation}
	As $a>0, b>0$, there holds the following estimates
	\begin{equation}\label{f1.713}
		\begin{split}
			\left| \int_{\Gamma_1}\frac{e^{-t}}{t-z}dt\right| =&\left|\int_{0}^{1} \frac{\ri e^{\ri \eta}}{\ri \eta+z}d\eta\right|
			\le\int_{0}^{1} \frac{1}{\sqrt{a+(\eta+b)^2}}d\eta\le\frac{1}{|z|}.\\
			\left| \int_{\Gamma_2}\frac{e^{-t}}{t-z}dt\right|=&\left|\int_{0}^{a+1} \frac{e^{\ri} e^{- \eta}}{ \eta-\ri-P}d\eta\right| 
			\le\int_{0}^{a+1} \frac{1}{\sqrt{(\eta-a)^2+(b+1)^2}}d\eta\\
			=&\ln{\frac{a+\sqrt{a^2+(b+1)^2}}{-1+\sqrt{1+(b+1)^2}}}
			\le 3|z|+1,\\
			\left| \int_{\Gamma_3}\frac{e^{-t}}{t-z}dt\right|=&\left|\ri e^{-(a+1)}\int_{-1}^{0} \frac{ e^{-\ri \eta}}{ 1+\ri(\eta-b)}d\eta\right|
			\le e^{-1}\int_{0}^{1} \frac{1}{ \sqrt{1+(\eta+b)^2}}d\eta\\
			\le&e^{-1}\int_{0}^{1} \frac{1}{ \sqrt{1+\eta^2}}d\eta=\frac{\ln{(1+\sqrt{2})}}{e}\le1.
		\end{split}
	\end{equation}
	\begin{figure}[h!]
		\centering
		\includegraphics[width=0.4\textwidth]{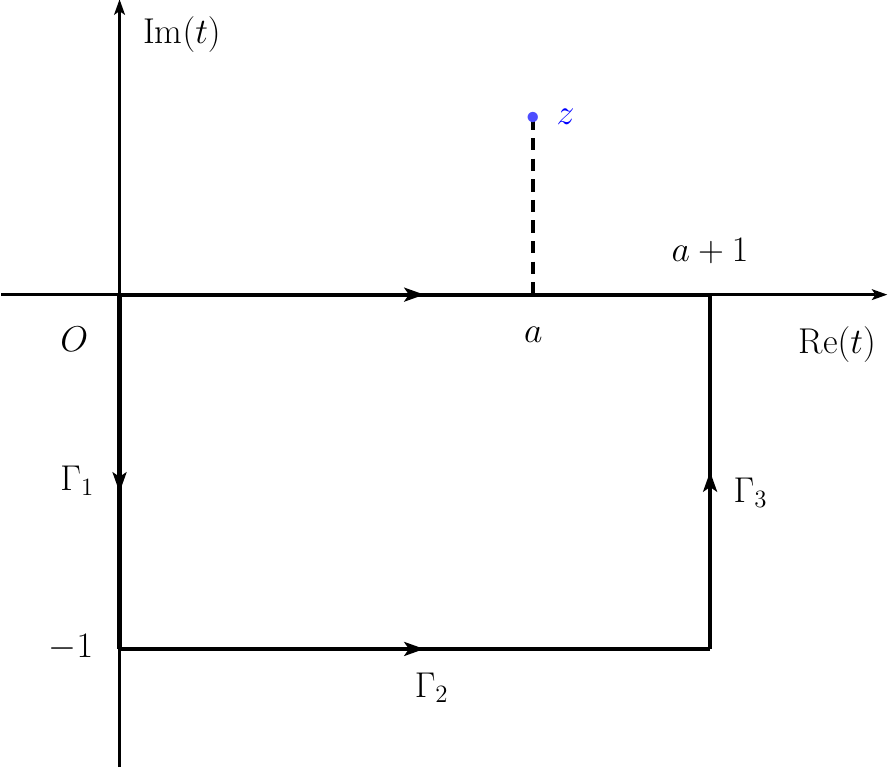}\qquad
		\includegraphics[width=0.4\textwidth]{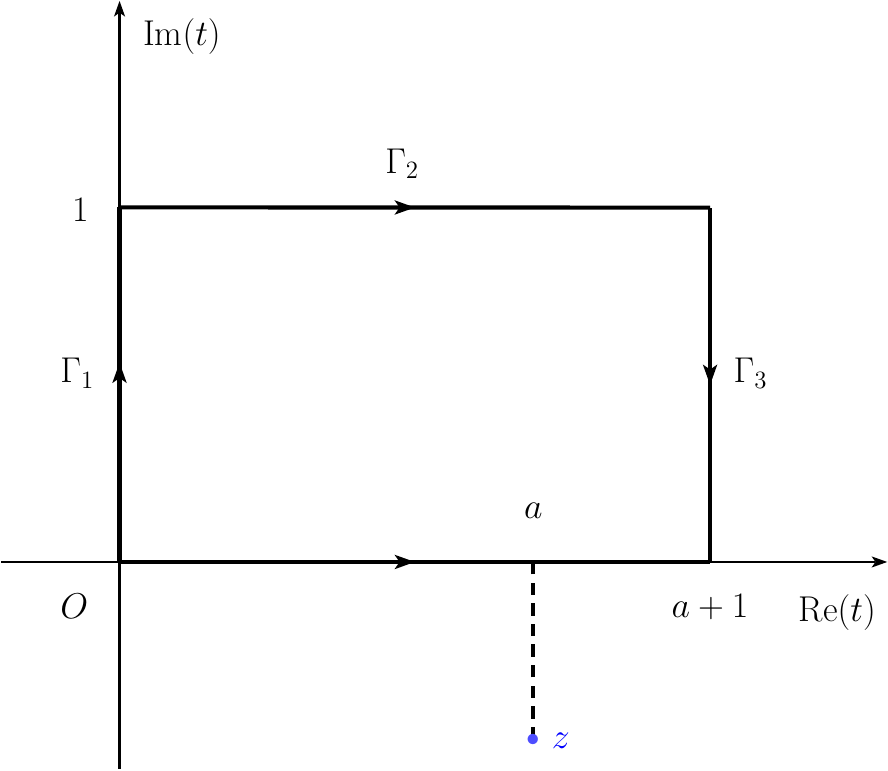}
		\caption{The integration path for $b>0$(left) and $b<0$(right)}
		\label{f1.718}
	\end{figure}
	Together with \eqref{integraldecompose} and \eqref{f1.717}, we obtain
	\begin{equation}
		\label{f1.716}	\left|\int_{0}^{+\infty}\frac{e^{-t}}{t-z}dt\right|\le\frac{1}{|z|}+3|z|+3.
	\end{equation}
	
	According to Stirling's formula \cite{Stirling1955}
	\begin{eqnarray}
		\nonumber n!=\sqrt{2\pi n}\Big(\frac{n}{e}\Big)^ne^{\frac{\theta_n}{12n}},  \quad 0<\theta_n<1,
	\end{eqnarray}
	for all $n \geq |z|e$, we can have
	\begin{eqnarray}
		\label{f1.71}n!\ge \sqrt{2\pi n}\Big(\frac{n}{e}\Big)^n\ge |z|^n ,
	\end{eqnarray}
	which further implies
	\begin{equation}\label{ppowerest}
		\sum\limits_{k=1}^{n}|z|^{n-k}(k-1)!\leq \sum\limits_{k=1}^{n}(n-k)!(k-1)!\leq 3(n-1)!.
	\end{equation}
	Substituting the above estimates \eqref{f1.716} and \eqref{ppowerest} into \eqref{integralestcase1}, we have 
	\begin{equation}
		\begin{split}
			\Big|\int_{0}^{+\infty}\frac{e^{-t}t^{n}}{t-z}dt \Big|
			\leq & 3(n-1)!+(n-1)!(1+3|z|^2+3|z|)\\
			\leq & (n-1)!(3|z|^2+3|z|+4).
		\end{split}
	\end{equation}
	Therefore, we finish the proof for the case $a, b> 0$. For the proof of the case $a>0$, $b<0$, we can just change the contour to $\Gamma_1\cup\Gamma_2\cup \Gamma_3$ as depicted in Fig. \ref{f1.718} (right) and then mimic the proof above. 
	
	Now, we consider the case $a>0$, $b=0$. In this case, the integrand has a pole at $t=z$ and the integral is defined along the contour $\Gamma_{\delta,1}\cup \Gamma_{\delta}$
	where
	\begin{equation}\label{f1.12}
		\Gamma_{\delta,1}=[0, a-\delta]\cup [a+\delta,+\infty),\quad
		\Gamma_{\delta}=\left\{t=\delta e^{\ri \theta}+a :-\pi\le \theta \le 0 \right\},
	\end{equation}
	see Fig. \ref{f1.4} (c) for an illustration. For the integral along $\Gamma_{\delta,1}$, we have
	\begin{equation}\label{f1.13} 
		\int_{\Gamma_{\delta,1}}\frac{e^{-t}t^{n}}{t-a}dt
		=\sum\limits_{k=1}^na^{n-k}(k-1)!+ a^n{\rm p.v.}\int_{0}^{+\infty}\frac{e^{-t}}{t-a}dt
	\end{equation}
	and
	\begin{equation}
		\begin{split}
			\Big|{\rm p.v.}\int_{0}^{+\infty}\frac{e^{-t}}{t-a}dt\Big|=e^{-a}\Big|{\rm p.v.}\int_{-a}^{a}\frac{e^{-\xi}}{\xi}d\xi+\int_{a}^{+\infty}\frac{e^{-\xi}}{\xi}d\xi\Big|\leq 2+\frac{1}{a}.
		\end{split}
	\end{equation}
	Then, apply the Stirling formula again, we obtain
	\begin{equation}\label{m1.0}
		\Big|\int_{\Gamma_{\delta,1}}\frac{e^{-t}t^{n}}{t-a}dt\Big|\leq 4(n-1)!+2a(n-1)!,
	\end{equation}
	for all $n\geq |z|e$. 
	By direct calculation leads to
	\begin{equation}\label{f1.14}
		\Big|\int_{\Gamma_{\delta}}\frac{e^{-t}t^{n}}{t-z}dt\Big|
		=\Big|\lim\limits_{\delta \rightarrow0^+}\int_{-\pi}^{0}e^{-(\delta e^{\ri\theta}+a)}(\delta e^{\ri\theta}+a)^{n}\ri d\theta\Big|
		=\pi e^{-a}a^n\le \pi a(n-1)!.
	\end{equation}
	Combining the above results\eqref{m1.0}-\eqref{f1.14} yields that
	\begin{equation}
		\begin{split}
			\Big|\int_{0}^{+\infty}\frac{e^{-t}t^{n}}{t-z}dt \Big|
			\leq & (\pi a+2a+4)(n-1)!.
		\end{split}
	\end{equation}
	This completes the proof of Lemma \ref{lemma1}.
\end{proof}

\begin{acknowledgements}
The author acknowledges the support of Clements Chair for this research.
This research is supported in part by funds from the Postgraduate Scientific Research Innovation Project of Hunan Province (No. CX20230512), NSFC (grant 12022104, 12371394, 12201603) and the Major Program of Xiangjiang Laboratory (No.22XJ01013). 
The authors would like to express their most sincere thanks to the referees and editors for their very helpful comments and suggestions, which greatly improved the quality of this paper.
\end{acknowledgements}

\end{document}